\documentclass[reqno,a4paper,12pt]{amsart}
\usepackage{amsmath,amsthm,amsfonts,amssymb, mathrsfs}
\usepackage{verbatim, graphicx, ifthen, enumitem}
\usepackage[T1]{fontenc}
\usepackage [applemac] {inputenc}
\usepackage{color}

\allowdisplaybreaks

\let\oldmarginpar\marginpar
\renewcommand\marginpar[1]{\-\oldmarginpar[\raggedleft\footnotesize #1]%
{\raggedright\footnotesize #1}}
\newtheorem{theorem}{Theorem}[section]
\newtheorem*{theorem*}{Theorem}

\newtheorem{lemma}[theorem]{Lemma}

\newtheorem{proposition}[theorem]{Proposition}
\newtheorem{corollary}[theorem]{Corollary}
\theoremstyle{definition}
\newtheorem{definition}[theorem]{Definition}
\newtheorem{example}[theorem]{Example}

\newtheorem{remark}[theorem]{Remark}
\newtheorem{thm}{Proposition}
\newenvironment{thmbis}[1]
  {\renewcommand{\thethm}{\ref{#1}$'$}%
   \addtocounter{thm}{-1}%
   \begin{thm}}
  {\end{thm}}
\newtheorem{thmm}{Theorem}
\newenvironment{thmmbis}[1]
  {\renewcommand{\thethmm}{\ref{#1}$'$}%
   \addtocounter{thm}{-1}%
   \begin{thmm}}
  {\end{thmm}}

\newcommand{\Z}{\mathbb{Z}}

\newcommand{\N}{\mathbb{d}}
\newcommand{\R}{\mathbb{R}}
\newcommand{\T}{\mathbb{T}}

\def\L{\Lambda}
\def\l{\lambda}
\def\T{\mathbb{T}}
\def\N{\mathbb{N}}
\def\Z{\mathbb{Z}}

\def\R{\mathbb{R}}
\def\R{\mathbb{R}}

\def\F{\mathcal{F}}

\def\1{\mathbf{1}}

\title[ Riesz basis bound]{On the lower Riesz basis bound of exponential systems over an interval}

\author{Thibaud Alemany}

\author{Shahaf Nitzan$^1$}
\thanks{\textcolor{white}{l}$^1$ The second author is supported by NSF CAREER grant DMS 1847796. This grant provided some support also to the first author}

\begin{document}

\begin{abstract}
We revisit Pavlov's characterization for Riesz bases of exponentials and study the corresponding lower Riesz basis bounds. In particular, this approach allows us to improve on known estimates for the bounds in Avdonin's theorem regarding average perturbations, and Levin's theorem regarding zeroes of sine-type functions.
\end{abstract}

\maketitle
\section{Introduction}

Let $\L\subseteq\R$ be a \textit{uniformly discrete} sequence, that is
\begin{equation}\label{seperated}
|\l-\l'|\geq \delta>0\qquad \forall\l\neq \l'\in\L,
\end{equation}
and consider the corresponding exponential system
\[
E(\L):=\{e^{2\pi i\l t}\}_{\l\in\L}.
\]
We say that this system is a \textit{Riesz basis} in $L^2[0,1]$ if there exist constants $A,B>0$ so that every function in the space admits a unique decomposition
\[
f=\sum_{\l\in\L} a_{\l}e^{2\pi i\l t}
\]
with
\[
A\sum_{\l\in\L}|a_{\l}|^2\leq \|f\|_{L^2[0,1]}^2 \leq B\sum_{\l\in\L}|a_{\l}|^2.
\]
The largest $A$ and smallest $B$ for which these inequalities hold are called the \textit{lower and upper Riesz basis bounds} respectively. We denote them by $A(\L)$ and $B(\L)$.

If $A(\L)=B(\L)=1$ then $E(\L)$ is an orthonormal basis in $L^2[0,1]$. It is straightforward to check that this happens if and only if $\L=\Z+\alpha$ for some $\alpha\in\R$. In the general case, the closer $A(\L)$ and $B(\L)$ are to $1$, the better are the properties of the basis: It provides series decompositions which converge at a faster rate and it is more stable to small perturbations. Estimating the Riesz basis bounds is therefore of interest. Estimates for $B(\L)$ are well known and not difficult to obtain (see e.g. (\ref{exp bessel}) below). One of the main goals of this paper is to provide corresponding estimates for the lower Riesz basis bound, $A(\L)$.

When considered over an interval, Riesz bases of exponentials are a classical topic of research. Their study goes back to Paley and Wiener who proved that if $|\l_n-n|\leq 1/\pi^2$, then $E(\L)=\{\l_n\}_{n\in\Z}$ is a Riesz basis in $L^2[0,1]$, \cite{PW, Wiener}. The sharp condition $|\l_n-n|\leq\mu<1/4$ was found by Kadetz, \cite{Kadetz} (See also \cite{Youngbook}). Kadetz' proof implies that in this case the corresponding lower Riesz basis bound satisfies
\[
A(\L)\geq 2\sin ^2{\frac{\pi}{4}(1-4 \mu)}.
\]

 A surprising extension of Kadetz' theorem was obtained by Avdonin, who showed that it is enough to consider stability under averaged perturbations, \cite{Avdonin}. More precisely, Avdonin's theorem states that if (\ref{seperated}) holds,
\[
L:=\sup_{n\in\Z}|\l_n-n|<\infty,
\]
and for some $N\in \N$ we have
\[
\mu^*(N):=\sup_{m\in\Z}\Big|\frac{1}{N}\sum_{n=mN}^{(m+1)N-1}(\l_n-n)\Big|<\frac{1}{4},
\]
then $E(\L)$ is a Riesz basis in $L^2[0,1]$.

In contrast to Kadetz' theorem, the lower Riesz basis bound in Avdonin's result does not directly follow from his proof. In Theorem \ref{thm:avdonin} we show that under the conditions of Avdonin's theorem the lower Riesz basis bound satisfies
\begin{equation}\label{intro avdonin}
A(\L)\geq \frac{1}{7\delta(\Lambda)}e^{- \frac{960\pi L^2N}{\delta(\Lambda)(1-4\mu^*)^2}}\sin^2{\frac{\pi}{4}(1-4\mu^*)},
\end{equation}
where $\mu^*=\mu^*(N)$. This improves on a previous estimate of Lindner \cite{Lindner}, by reducing the order of decay from a double-exponential to an exponential one. (See section \ref{kadetz lindner avdonin}).

Riesz bases of exponentials were identified not only via perturbation results. In a different direction of study, they were identified as zero sets of certain analytic functions. A well known family of such functions is the family of \textit{sine-type functions}. An analytic function $F$ belongs to this class if it is \textit{of exponential type} and there exist $y, m, M>0$ so that
\begin{equation}\label{bounded from above and below I}
m\leq |F(x+iy)|^2 \leq M \qquad  x \in\R.
\end{equation}
In  \cite{Levin, Golovin} Levin and Golovin prove that the zero set of a sine-type function  of \textit{full diagram} gives a Riesz basis in $L^2[0,1]$ (See Section \ref{sine type} for a more detailed discussion of these conditions). In Corollary \ref{sine type}
we find that for $\L\subseteq \R$, under these conditions, the lower Riesz basis bound satisfies
\begin{equation}\label{intro sine type}
A(\Lambda)\geq \frac{1}{7\delta(\L)}\frac{m}{M}e^{-\frac{8\pi y}{\delta(\L)}},
\end{equation}
where $y,m,M>0$ are as in (\ref{bounded from above and below I}).

These research directions culminated with a celebrated theorem of Pavlov, which gives a full characterization of exponential Riesz bases over an interval \cite{Pav}. Roughly speaking, Pavlov proved that if the condition in (\ref{bounded from above and below I}) is replaced with the weaker condition, that the weight  $w^{(y)}(x):=  |F(x+iy)| ^2$ satisfies
\[
\sup_{I=[a,b]\subset\R}\frac{1}{|I|^2}\int_Iw^{(y)}\int_I\frac{1}{w^{(y)}}<\infty,
\]
 then the sufficient condition in the Levin-Golovin theorem becomes both sufficient and necessary. Following this, in  \cite{Hrusch},  Hrusch\"{e}v showed that the perturbation theorems formulated above can be obtained from Pavlov's characterization. A comprehensive survey, including several additional results, was given by Pavlov, Hrusch\"{e}v, and Nikolskii \cite{HNP}.

Here we return to Pavlov and Hrusch\"{e}v's approach with two goals in mind: To give what we believe is a relatively straightforward presentation of it, and to do so while studying how different elements in it may impact the lower Riesz basis bound. Pavlov's main idea builds on the observation that the Riesz basis property is equivalent to the invertibility of  a restricted projection operator, denoted below by $P_{K_S}|K_{B}$. He then shows that this operator is invertible if and only if the Riesz projection $R$ is bounded in a certain weighted space $L^2_w(\R)$. We present a different proof for the latter equivalence. In particular it directly gives the relation
\[
\|(P_{K_S}|K_{B})^{-1}\|=\|R\|_{L^2_w(\R)},
\]
(See Theorem \ref{rb as riesz} below). This allows us to relate the norm of the Riesz projection and the corresponding lower Riesz basis bound.

In \cite{HNP}, Pavlov, Hrusch\"{e}v, and Nikolskii give some additional versions of Pavlov's characterization, formulated in terms of certain phase functions and in terms of the counting function of $\L$. We consider these results as well, and give corresponding estimates for the lower Riesz basis bounds. With this we obtain the estimates (\ref{intro avdonin}) and  (\ref{intro sine type}) formulated above. Finally, we consider also applications to Marcinkiewicz-Zygmund families. In particular, we show that a similar bound holds for Avdonin's theorem also in this setting.

\begin{remark}
To simplify the presentation, we restrict our attention to real sequences $\L$. The estimates we give for $A(\L)$ can be extended without much change to the case where $\L$ lies in a strip parallel to the real axis. Pavlov's original proof considered only this case. The case where $\L$  lies in a half plain was treated by Nikolskii in \cite{Nikbook} (see also \cite{HNP}), and the general case was discussed by Minkin in \cite{Min}. To estimate the lower Riesz basis bound in these cases one needs first to consider the bounds in the so called \textit{Carleson condition}, we do not enter this discussion here.
\end{remark}

\begin{remark}
In \cite{SeipLiubAp} Lyubarskii and Seip take a different approach to the problem, which allows them to extend the result also to $p$-normed spaces. An estimate of the lower Riesz basis bound may be obtained using this approach as well, but the direct relation formulated in  Theorem \ref{rb as riesz} does not follow from it. Some discussion of Lyubarskii and Seip's presentation of the result appears in Section \ref{charact-generat}.
\end{remark}

The presentation below is organized as follows. In Section 2 we describe the equivalence between the Riesz basis problem, the invertibility of the operators $P_{K_S}|K_{B}$, and the boundedness of the Riesz projection in $L^2_w(\R)$. We add a short survey of known results related to the latter question. We apply this machinery in Section 3 where we discuss characterizations of Riesz bases and estimate the corresponding lower Riesz basis bounds. In Section 4 we look at some applications. In particular, we give bound estimates in Avdonin's theorem and the Levin-Golovin theorem. We consider Marcinkiewicz-Zygmund families at the end of Section 4.
\section{Riesz bases and the Riesz projection}

In this section we recall the basic properties of Riesz bases. Next, we follow the approach from \cite{Pav} and observe a relation between exponential Riesz bases and the Riesz projection in appropriate weighted spaces. Finally we shortly survey results related to norm estimates for the latter operator.

\subsection{Riesz bases in the abstract setting}\label{section abstract setting}

Throughout this subsection let $H$ be a separable Hilbert space and  let $I$ be a countable index set. We denote by $\{e_n\}_{n\in I}$ the standard orthonormal basis in $\ell^2(I)$, that is, for $k,n\in I$ we have $e_n(k)=1$ if $k=n$ and $e_n(k)=0$ otherwise. We are interested in the following.

\begin{definition}\label{def rb}
Let $\Phi:=\{\phi_n\}_{n\in I}\subset H$. We say that $\Phi$ is a Riesz basis in $H$ if there exists a bounded invertible operator $T:\ell^2(I)\rightarrow H$ which satisfies $Te_n=\phi_n$ for every $n\in I$.
\end{definition}
The condition in Definition \ref{def rb} holds if and only if $\Phi$ is complete in $H$ and
\begin{equation}\label{rb def 2}
A\sum_{n\in I}|a_n|^2\leq \|\sum_{n\in I} a_n\phi_n\|^2\leq B\sum_{n\in I}|a_n|^2,
\end{equation}
for any finite sequence $\{a_n\}\in \ell^2(I)$. The largest $A$ and smallest $B$ for which this inequality holds are called the \textit{lower and upper Riesz basis bounds}. We denote them by $A(\Phi)$ and $B(\Phi)$. A system $\Phi$ which satisfies the inequality in (\ref{rb def 2}) but is not necessarily complete is called a \textit{Riesz sequence}.

It is strait forward to check that Definition \ref{def rb} coincides with the presentation of Riesz bases as was given in the introduction, that is, that every $f\in H$ admits a unique decomposition
\[
f=\sum_{n\in I} a_n\phi_n
\]
with
\[
A\sum_{n\in I}|a_n|^2\leq \|f\|^2\leq B\sum_{n\in I}|a_n|^2.
\]
Note that if $T$ is the operator from Definition \ref{def rb} then $A(\Phi)=\|T^{-1}\|^{-2}$ and $B(\Phi)=\|T\|^2$.

By considering the adjoint of $T$, one may check that $\Phi$ is a Riesz basis in $H$ if and only if it is a minimal system and
\begin{equation}\label{frame def 2}
 A\|f\|^2\leq \sum|\langle f,\phi_n\rangle|^2\leq B\|f\|^2\qquad\forall f\in H,
\end{equation}
where $B=\|T^*\|^2=\|T\|^2=B(\Phi)$ and $A=\|(T^*)^{-1}\|^{-2}=\|T^{-1}\|^{-2}=A(\Phi)$.
A system $\Phi$ which satisfies the inequality in (\ref{frame def 2}) but is not necessarily minimal is called a \textit{frame}. In particular, it follows that the right hand inequalities in  (\ref{rb def 2}) and (\ref{frame def 2}) are equivalent. A system which satisfies these inequalities, but not necessarily the left hand side inequalities, is called a \textit{Bessel system}.

Assume that $\Phi$ is a Riesz bases in $H$. Since Riesz bases are both complete and minimal, $\Phi$ admits a unique dual system $\{g_n\}\subseteq H$ so that $\langle\phi_n,g_k\rangle =1$ if $k=n$ and $\langle\phi_n,g_k\rangle =0$ otherwise.
We therefore have
\begin{equation}\label{dual riesz}
f=\sum_{n\in I}\langle f,g_n\rangle \phi_n \qquad\forall f\in H.
\end{equation}

The following lemma is an abstract version of a lemma which can be found in \cite{HNP} (see pp 218-219 for a short survey of works where it appears). This lemma gives a necessary and sufficient condition for the Riesz basis property to be preserved under projections.

\begin{lemma}\label{lemma: rb and projection}
Let $\Phi:=\{\phi_n\}_{n\in I}\subseteq H$. Assume that $\Phi$ is a Riesz basis in $S:=\overline{\textrm{span}}\:(\Phi)$. Given a closed subspace $L\subseteq H$ let $P_L:H\rightarrow L$ denote the orthogonal projection of $H$ onto $L$. Then the following are equivalent:
\begin{itemize}
\item[i.] The system $P_L\Phi:=\{P_L\phi:\:\phi\in \Phi\}$ is a Riesz basis in $L$.

\vspace{5pt}

\item[ii.] The restriction $P_L|S:S\rightarrow L$ is an isomorphism of $S$ onto $L$.
\end{itemize}
Moreover, in this case we have
\begin{equation}\label{rb bound relations}
A(\Phi)\|(P_L|S)^{-1}\|^{-2}\leq A(P_L\Phi)\quad\textrm{and}\quad B(P_L\Phi)\leq B(\Phi)\|(P_L| S)\|^2.
\end{equation}
\end{lemma}
\begin{proof}
Assume that $(i)$ holds. As $\Phi$ and $P_L\Phi$ are Riesz bases in $S$ and $L$ respectively, there exist isomorphisms
$T_S:\ell^2(I)\rightarrow S$ and $T_L:\ell^2(I)\rightarrow L$
which satisfy $T_Se_n=\phi_n$ and $T_Le_n=P_L\phi_n$ for every $n\in I$. It follows that
$
P_L\phi_n=T_LT_S^{-1}\phi_n
$. Since $\Phi$ spans $S$ we may conclude that
$
P_L|S=T_LT_S^{-1},
$
and so that $P_L|S$ is also an isomorphism. The opposite direction may be obtained in a similar way.
Next, recall that $\|T_S^{-1}\|^{-2}=A(\Phi)$ and $\|T_L^{-1}\|^{-2}=A(P_L\Phi)$. The relation $T_L=P_LT_S$ may therefore be applied to obtain the first estimate in \ref{rb bound relations}. The second estimate is obtained in a similar way.
\end{proof}

\subsection{First observations regarding Riesz bases of exponentials}\label{shift by y}

Let $\Lambda\subseteq\R$ and consider the system
\[
E(\Lambda):=\{e^{2\pi i\lambda t}\}_{\lambda\in\Lambda}.
\]
We are interested in cases where $E(\Lambda)$ is a Riesz basis in $L^2[0,1]$. Our goal is to estimate the Riesz basis bounds in such cases. For brevity we denote them by $A(\Lambda):=A(E(\Lambda))$ and $B(\Lambda):=B(E(\Lambda))$.

Observe that if $E(\Lambda)$ is a Riesz basis in $L^2[0,1]$ then $\L$ is uniformly discrete (recall (\ref{seperated})). Indeed, to see this apply (\ref{rb def 2}) with $a_{\l}=a_{\l'}=1$ and all other coefficients equal $0$. In what follows we may therefore assume that $\L$ is uniformly discrete, as we will do throughout. The largest $\delta>0$ for which  (\ref{seperated}) holds is called the \textit{separation constant} of $\Lambda$. We denote it by $\delta(\Lambda)$.

It is well known that if $\L$ is uniformly discrete then $E(\Lambda)$ is a Bessel system. In this case the right hand inequality in (\ref{rb def 2}) holds with the estimate
\begin{equation}\label{exp bessel}
B(\Lambda)\leq \frac{8\pi}{\min\{\delta(\Lambda), 1\}},
\end{equation}
(see e.g. \cite{OlUlbook}, Proposition 2.7). We therefore focus our attention on the lower Riesz basis bound, $A(\Lambda)$.

Following Pavlov, our first step will be to apply Lemma \ref{lemma: rb and projection} with $L^2(\R^+)$ and $L^2[0,1]$ playing the roles of $H$ and $L$ respectively (here $\R^+:=[0,\infty]$). As $E(\L)\nsubseteq L^2(\R^+)$ we apply Lemma \ref{lemma: rb and projection} with an auxiliary systems of the form $E(\L_y)$, where
\[
\Lambda_y:=\Lambda+iy, \quad y>0.
\]
One can reedily check that $E(\L)$ is a Riesz basis in $L^2[0,1]$ if and only if the same is true for $E(\L_y)$ and moreover, that
\begin{equation}\label{L and Ly}
A(\L_y)\leq A(\L)\qquad \forall y>0.
\end{equation}

Denote the closed span of $E(\L_y)$ in $L^2(\R^+)$ by $\mathcal{E}(\L_y)$.
We first estimate the lower Riesz basis bound of $E(\L_y)$ in $\mathcal{E}(\L_y)$.

\begin{lemma}\label{lemma: ingham}
Let $\L\subseteq\R$ be a uniformly discrete sequence and let $y>0$. Then the system $E(\L_y)$ is a Riesz basis in $\mathcal{E}(\L_y)$ and its lower Riesz basis bound is no smaller than
\[
 \frac{1}{7\delta(\Lambda)}e^{-{8\pi y}/{\delta(\Lambda)}}.
\]
\end{lemma}

\begin{proof}
Given $a>0$, it is straightforward to verify that $E(\Lambda)$ satisfies (\ref{rb def 2}) over the interval $[0,a]$ if and only if the same is true for $E(\Lambda_y)$. Moreover, we have
\begin{equation}\label{riesz on a}
A_a(\Lambda)e^{-4\pi ya}\leq A_a(\Lambda_y),
\end{equation}
where $A_a$ denotes the lower bound in (\ref{rb def 2}) when the systems are considered over $[0,a]$ (so that $A_1(\L)=A(\L)$).

Fix $a=2/\delta(\Lambda)$. It follows from Ingham's theorem (see e.g. \cite{Youngbook}, p 136) that $E(\L)$ satisfies (\ref{rb def 2}) over the interval $[0,a]$ and
\[
A_a(\Lambda)\geq \frac{a}{\pi^2}(1-\frac{1}{(a\delta(\Lambda))^2})\geq \frac{1}{7\delta(\Lambda)},
\]
 Combined with the estimate in (\ref{riesz on a}) we find that for $y>0$, the system $E(\L_y)$  satisfies (\ref{rb def 2}) over the interval $[0,a]$ and
  \[
\frac{1}{7\delta(\Lambda)}e^{-8\pi y/\delta(\Lambda)}\leq A_a(\Lambda_y).
\]
It follows that $E(\L_y)$ satisfies (\ref{rb def 2}) also over $\R^+$, and that the corresponding lower bound satisfies the required estimate.
\end{proof}

Lemmas \ref{lemma: rb and projection} and \ref{lemma: ingham}, combined with the estimate (\ref{L and Ly}), give the following.

\begin{proposition}\label{thm; exp rb and projections}
Let $\Lambda\subseteq\R$ be uniformly discrete and denote by $P_{L^2[0,1]}$ the orthogonal projection from $L^2(\R^+)$ onto $L^2[0.1]$. Then, $E(\Lambda)$ is a Riesz basis in $L^2[0,1]$ if and only if there exists $y>0$ so that the restricted projection $P_{L^2[0,1]}|\mathcal{E}(\L_y):\mathcal{E}(\L_y)\rightarrow L^2[0,1]$ is an isomorphism. Moreover, in this case these restricted projections are isomorphisms for all $y>0$, and we have
\begin{equation}\label{eq: riesz bound estimate abstract}
A(\Lambda)\geq \sup_y  \frac{1}{7\delta(\Lambda)}e^{-{8\pi y}/{\delta(\Lambda)}}\big\|\big(P_{L^2[0,1]}|\mathcal{E}(\L_y)\big)^{-1}\big\|^{-2}.
\end{equation}
\end{proposition}

Let $-\Lambda:=\{-\l:\:\l\in\L\}$. It is straightforward to check that $E(\L)$ is a Riesz basis in $L^2[0,1]$ if and only if $E(-\Lambda)$ is, and that $A(\L)=A(-\L)$. We may therefore write  $\mathcal{E}((-\L)_y)$ instead of $\mathcal{E}(\L_y)$ in the estimate (\ref{eq: riesz bound estimate abstract}). We will do so for technical reasons, which will be made clear in the next subsection.

\subsection{A reformulation in terms of model spaces}
	
We denote the inverse Fourier transform by $\F^*$ and use the normalization
\[
\F^*f(z)=\int_{\R}f(t)e^{2\pi i zt}dt\qquad f\in L^1(\R),
\]
with the standard extension to $L^2(\R)$. With this normalization $\F^*$ is an isometry from $L^2(\R)$ onto itself. The Paley-Wiener theorem asserts that the image of $L^2(\R^+)$ under this isometry is the Hardy space over the upper half plane, $\mathbf{H}^2:=\mathbf{H}^2(\mathbb{C}^+)$, consisting of all function $f$ which are analytic in the upper half plane and satisfy
\[
\sup_{y>0}\int_{\R}|f(x+iy)|^2dx<\infty.
\]
We refer the reader to e.g. \cite{Garbook, Koosis} for basic information related to Hardy spaces.

 To evaluate the operator norms in (\ref{eq: riesz bound estimate abstract}) (after replacing  $\mathcal{E}(\L_y)$ by $\mathcal{E}((-\L)_y)$ as described thereafter), we apply the inverse Fourier transform to the corresponding subspaces involved. Denote $S(z):=e^{2\pi iz}$, we keep this notation fixed throughout. Since the Fourier transform replaces translations with modulations we have
  \[
  \F^*L^2[1,\infty]= S \mathbf{H}^2.
   \]
  As $\F^*$ is an isometry, it preserves orthogonal complements and so
\[
\F^*L^2[0,1]=K_S:=  \mathbf{H}^2\ominus S \mathbf{H}^2.
\]
Here $K_S$ is the so called \textit{model space} generated by the inner function $S$.

It turns out that the image of $\mathcal{E}((-\L)_y)$ under $\F^*$ has a similar structure. To see this first note that since,
\[
\F^*f(\l+iy)=\langle f, e^{2\pi i (-\l+iy)t}\rangle\qquad f\in L^2(\R^+),
\]
a function $f\in L^2(\R^+)$ is orthogonal to $\mathcal{E}((-\L)_y)$ if and only if $\F^*f$ vanishes on $\L_y$. Since $\Lambda$ is uniformly discrete, the sequence $\L_y$ satisfies the \textit{Blaschke condition}
\[
\sum_{\l\in\L}\frac{y}{\l^2+y^2}<\infty,
\]
 and  so there exist unimodular constants $\varepsilon_{\l}$ so that the \textit{Blaschke product}
\[
B_{\L_y}(z):=\prod_{\lambda\in\Lambda}\varepsilon_{\l}\frac{1-\frac{z}{\l+iy}}{1-\frac{z}{\l-iy}}
\]
converges uniformly on compact sets in the upper half plane  (see e.g.  \cite{Garbook, Koosis}). We conclude that $f$ is orthogonal to $\mathcal{E}((-\L)_y)$ if and only if $\F^*f$ is of the form $\F^*f=B_{\L_y}g$ for some $g\in \mathbf{H}^2$. This can be rewritten as
\[
 \F^*\Big(\mathcal{E}((-\L)_y)^{\perp}\Big)=B_{\Lambda_y}\mathbf{H}^2.
 \]
Recalling that $\F^*$ is an isometry, we find that
\[
\F^*\mathcal{E}((-\L)_y)=K_{B_{\L_y}}:= \mathbf{H}^2\ominus B_{\L_y}\mathbf{H}^2.
\]
Here $K_{B_{\L_y}}$ is the model space generated by the inner function ${B_{\L_y}}$.

With these relations established, Proposition \ref{thm; exp rb and projections} may be reformulated as follows, where we denote the projection of $\mathbf{H}^2$ onto $K_S$ by $P_S:=P_{K_S}$ for brevity.
\begin{thmbis}{thm; exp rb and projections}\label{thm; exp rb modal}
\textit{Let $\Lambda\subseteq\R$ be a uniformly discrete sequence. Then, $E(\Lambda)$ is a Riesz basis in $L^2[0,1]$ if and only if there exists $y>0$ so that the restricted projection $P_S|K_{B_{\L_y}}:K_{B_{\L_y}}\rightarrow K_S$ is an isomorphism. Moreover, in this case these restricted projections are isomorphisms for all $y>0$, and we have}
\begin{equation}\label{eq: riesz bound estimate modal}
A(\Lambda)\geq \sup_y \frac{1}{7\delta(\Lambda)}e^{-{8\pi y}/{\delta(\Lambda)}}\big\|\big(P_S|K_{B_{\L_y}}\big)^{-1}\big\|^{-2}.
\end{equation}
\end{thmbis}

\subsection{A reformulation in terms of the Riesz projection}\label{Riesz projection subsection}
It turns out that the operator norm which appears in (\ref{eq: riesz bound estimate modal}) is equal to the norm of a certain weighted Riesz projection. To give an explicit formulation, we first define the Riesz projection as an operator in weighted $L^2$ spaces.

The presentation of the Riesz projection over the torus is simple: The image of a trigonometric polynomial is obtained by changing the coefficients of exponentials with negative frequencies to zero. We present the Riesz projection over $\R$ in a similar way, after applying the canonical isometry between the corresponding Hardy spaces. This (somewhat cumbersome) way to present the Riesz projection allows us to give a basic proof of Theorem \ref{rb as riesz} below, while doing so over $\R$. In what follows, as customary, we identify functions analytic in the upper (or lower) half plane with their boundary values on $\R$. We use the convention $\Pi(x):=1/(1+x^2)$ and say that $f\in L^p_{\Pi}(\R)$ if
\[
\int_{\R}\frac{|f(x)|^p}{1+x^2}\: dx <\infty.
\]

Let $\mathbb{D}$ be the unite disc, and let $\mathbf{H}^2(\mathbb{D})$ be the corresponding Hardy space. Consider the function $\phi:\mathbb{C}^+\rightarrow \mathbb{D}$ given by
 \[
 \phi(z)=\frac{i-z}{i+z}.
 \]
It is well known that the operator $T:\mathbf{H}^2(\mathbb{D})\rightarrow \mathbf{H}^2$, defined by
 \[
 (TF)(z)=\frac{1}{\sqrt{\pi}}\frac{F(\phi(z))}{i+z},
 \]
is an isometric isomorphism (see e.g. \cite{Garbook} page 52). For $n\in\N$ the monomials $Q_n(w)=w^{n-1}\in \mathbf{H}^2(\mathbb{D})$ satisfy $T(Q_n)=r_n$, where
  \[
 r_n(z)=\frac{(i-z)^{n-1}}{\sqrt{\pi}(i+z)^{n}}\qquad n\in\N.
 \]
 Denote $M:=\textrm{span}\{r_n:\:n\in\N\}$.
Since $T$ is an isometry, $M$ is dense in $\mathbf{H}^2$.

In a similar way we may conclude that if $h$ is outer in the upper half plane and $h\in L^2_{\Pi}(\R)$ then $Mh$ is dense in $\mathbf{H}^2$. Indeed, the condition on $h$ implies that there exists an outer function $H\in \mathbf{H}^2(\mathbb{D})$ so that $H(\phi(z))=h(z)$ (see e.g. \cite{Nikbook} page 83). It follows that $T(Q_nH)=r_nh$. Since $H$ is outer, Beurling's theorem implies that the set $\{Q_nH:\:n\in \N\}$ is complete in  $\mathbf{H}^2(\mathbb{D})$  (see e.g. \cite{Koosis}). Since $T$ is an isometry the same holds for $\{r_nh:\:n\in \N\}$ in $\mathbf{H}^2$.

Let $\mathbf{H}^2_{-}$ denote the orthogonal complement of  $\mathbf{H}^2$ in $L^2(\R)$, so that $\mathbf{H}^2_{-}=\F^*L^2(\R^-)$ is the Hardy space in the lower half plane. For $f\in \mathbf{H}^2$ we use the convention $f^*(z)=\bar{f}(\bar{z})\in \mathbf{H}^2_{-}$. Given $n\in \N$ put
$
r_{-n}:=r^*_n(z)
$
 and let $N:=\textrm{span}\{r_{-n}:\:n\in\N\}$. Due to symmetry, we find that $N$ is dense in $\mathbf{H}^2_{-}$, and that if $h$ is outer in the upper half plane with $h\in L^2_{\Pi}(\R)$ then $\{Nh^*\}$ is dense there as well. In particular, this discussion implies that $E:=M+ N$ is dense in $L^2(\R)$.
 
Applying similar considerations one may show that if $f\in L^1(\R)$ and
 \begin{equation}\label{weighted L1}
 \int_{\R}f(x)\Big(\frac{i-x}{i+x}\Big)^n\: dx=0,\qquad \forall n\in\Z,
 \end{equation}
 then $f=0$ almost everyehere. (This is the analog, under an appropriate isometry, of the fact that the Fourier transform is an injection on $L^1(\T)$).

Let $w\in L^1_{\Pi}(\R)$, $w> 0$ a.e.,  and consider the corresponding weighted space
   \[
 L^2_w(\R):=\{g:\: \|g\|_w^2:=\int_{\R}|g|^2wdx <\infty\}.
 \]
Observe that $E=M+N$ is a subset of $L^2_w(\R)$, and that it is dense there. Indeed, assume that $g\in  L^2_w(\R)$ is orthogonal to $M+N$ in the space. Since $g\in  L^2_w(\R)$ and  $w\in L^1_{\Pi}(\R)$, we have $gw/(x+i)\in L^1(\R)$. Put $f:=gw/(x+i)$. Then, since $g$ is orthogonal to $M+N$ in $L^2_w(\R)$ , $f$ satisfies (\ref{weighted L1}). It follows that $f=0$ almost everywhere. As $w\neq 0$ almost everywhere, we conclude that $g=0$ almost everywhere.

For a finite linear combination
\[f=\sum_{n=-K,n\neq 0}^K a_nr_n\in E, \] the Riesz projection is defined to be the projection on $M$ along $N$, that is,
$
Rf=\sum_{1}^{K}a_nr_n.
$
We say that the Riesz projection is bounded in $L^2_w(\R)$ if
\[
\|Rf\|_{L^2_w}\leq \|R\|_w\|f\|_{L^2_w} \qquad \forall f\in E,
\]
where $\|R\|_w$ is a positive constant. If this happens $R$ can be extended to a bounded linear operator on $L^2_w(\R)$. The following lemma is well known, we give a sketch of the proof for completeness.  

\begin{lemma}\label{1/h}
Let $w\in L^1_{\Pi}(\R)$, $w> 0$ a.e. If the Riesz projection is bounded in $L^2_w(\R)$ then $1/w\in L^1_{\Pi}(\R)$ and the Riesz projection is bounded in $L^2_{1/w}(\R)$ as well. In particular, this implies that there exists a function $h$, outer in the upper half plane, so that $w=|h|^2$.
\end{lemma}
\begin{proof}
 Multiplying by $(i+x)^n/(i-x)^n$, $n\in\Z$, if necessary, one can show that the Riesz projection is bounded in $L^2_w(\R)$  if and only if the partial sum operators
\[
S_K(\sum_{n=-L,n\neq 0}^L a_nr_n)=\sum_{n=-K,n\neq 0}^K a_nr_n
\]
are uniformly bounded in the space. This is equivalent to saying that $\{r_n,r_{-n}\}_{n\in\N}$ is a Schauder basis in $L^2_w(\R)$. In particular, it follows that $\{r_n,r_{-n}\}_{n\in\N}$ is minimal in the space (See e.g. \cite{Youngbook} for the definition of Schauder bases and for their basic properties). It is a simple exercise to check that the minimality property implies that $1/w\in L^1_{\Pi}(\R)$, and that the dual basis to $\{r_n,r_{-n}\}_{n\in\N}$ is $\{r_n/w,r_{-n}/w\}_{n\in\N}$. Since the dual basis is also a Schauder basis, it follows that the Riesz projection is bounded in $L^2_{1/w}(\R)$. Finally, as  $1/w\in L^1_{\Pi}(\R)$, we have  $\log w\in L^1_{\Pi}(\R)$. This implies that an outer function $h$ with $w=|h|^2$ exists.
\end{proof}

We are now ready to formulate a relation between the operator norm from (\ref{eq: riesz bound estimate modal}) and the norm of a weighted Riesz projection. The connection between the two operators has been established by Pavlov \cite{Pav}. The presentation here is different. In particular, it allows us to directly obtain the relation between the corresponding norms.

For an inner function $\Theta$ denote $K_{\Theta}= \mathbf{H}^2\ominus\Theta\mathbf{H}^2$ and let $P_{\Theta}$ be the orthogonal projection from $\mathbf{H}^2$ onto this space.

\begin{theorem}\label{rb as riesz}
Let $I$ and $\Theta$ be two inner functions in the upper half-plane. Assume that there exists an outer function $h$ so that on $\R$ we have $\overline{I}\Theta=\overline{h}/{h}$ and $h\in L^2_{\Pi}(\R)$. Denote $w=|h|^2$. Then the following are equivalent
\begin{itemize}
\item[i.] The restricted projection $P_{I}|K_{\Theta}$ is an isomorphism.

\vspace{5pt}

\item[ii.] The Riesz projection $R$ is a bounded operator on $L^2_w(\R)$.
\end{itemize}

Moreover, in this case we have $\|(P_{I}|K_{\Theta})^{-1}\|=\|R\|_{w}$.
\end{theorem}

\begin{proof}
It is straightforward to check that $P_{\Theta}|K_I=(P_{I}|K_{\Theta})^*$ and so the operators  $P_{\Theta}|K_I$ and $P_{I}|K_{\Theta}$ are invertible together, and their inverses have equal norm. For technical reasons it will be more convenient for us to prove the statement of the theorem with the operator  $P_{\Theta}|K_I$, as we will now proceed to do.

First, observe that multiplication by $\Theta$ is an isometry on  $L^2(\R)$ and so
\begin{equation}\label{theta isometry}
L^2(\R)=\Theta \mathbf{H}^2 \oplus \Theta \mathbf{H}^2_{-}.
\end{equation}
On the other hand, the definition of $K_{\Theta}$ implies that $L^2(\R)=\Theta\mathbf{H}^2 \oplus K_{\Theta}\oplus \mathbf{H}^2_{-}$. Consequently we have
\begin{equation}\label{model space identity}
 \Theta \mathbf{H}^2_{-}= K_{\Theta}\oplus \mathbf{H}^2_{-}.
\end{equation}

Let $\widetilde{P}_{ \Theta \mathbf{H}^2_{-}}:L^2(\R)\rightarrow \Theta \mathbf{H}^2_{-}$ be the orthogonal projection from $L^2(\R)$ onto $\Theta \mathbf{H}^2_{-}$. If $g\in L^2(\R)$ satisfies $g=g_1+g_2$, with $g_1\in \mathbf{H}^2$ and $g_2\in \mathbf{H}^2_{-}$, then the identity (\ref{model space identity}) implies that $\widetilde{P}_{ \Theta \mathbf{H}^2_{-}}g= P_{\Theta}g_1 +g_2$. It follows that $\|\widetilde{P}_{ \Theta \mathbf{H}^2_{-}}g\|^2= \|P_{\Theta}g_1\|^2 +\|g_2\|^2$. This, combined with the estimate $\frac{a}{b}\leq \frac{a+c}{b+c}$, which holds whenever $a\leq b$ and $c>0$, implies that
\begin{equation}\label{first identity}
\inf_{g\in K_I} \frac{\|P_{\Theta}g\|^2}{\|g\|^2}=\inf_{g\in K_{I}\oplus \mathbf{H}^2_{-}} \frac{\|\widetilde{P}_{ \Theta \mathbf{H}^2_{-}}g\|^2}{\|g\|^2}=\inf_{g\in I\mathbf{H}^2_{-}} \frac{\|\widetilde{P}_{ \Theta \mathbf{H}^2_{-}}g\|^2}{\|g\|^2},
\end{equation}
where the last equality is due to (\ref{model space identity}).

The identity (\ref{theta isometry}) and the definition of a projection allow us to rewrite (\ref{first identity}) as
\[
\inf_{g\in K_I} \frac{\|P_{\Theta}g\|^2}{\|g\|^2}=\inf_{g\in I\mathbf{H}^2_{-}, f\in \Theta \mathbf{H}^2} \frac{\|g+f\|^2}{\|g\|^2}.
\]

Recall that, as $h$ is outer and $h\in L^2_{\Pi}(\R)$, the sets $h^*N$ and $hM$ are dense in $\mathbf{H}^2_{-}$ and $\mathbf{H}^2$ respectively. The right hand side in the last equality can therefore be rewritten as
\[
\inf_{G\in N, F\in M} \frac{\|Ih^*G+\Theta h F\|^2}{\|h^*G\|^2}.
\]
 When restricted to the image of $P_{\Theta}|K_I$, the inverse operator therefore satisfies
\[
\|(P_{\Theta}|K_I)^{-1}\|^{2}=\sup_{G\in N, F\in M} \frac{\|h^*G\|^2}{\|Ih^*G+\Theta h F\|^2}.
\]
Recall that on $\R$ we have $h^*=\overline{h}$. Inserting the relation $\Theta\overline{I}= \overline{h}/{h}$ into the last expression we find that it equals
\[
\sup_{G\in N, F\in M} \frac{\|\overline{h}G\|^2}{\|\Theta hG+\Theta h F\|^2}=\sup_{G\in N, F\in M} \frac{\|G\|_w^2}{\|G+ F\|_w^2}=\|id-R\|^2_w=\|R\|^2_w.
\]

The above consideration proves that the Riesz projection is bounded if and only if $P_{\Theta}|K_I$ is bounded from below, and that the corresponding norms are equal. It remains to show that $P_{\Theta}|K_I$ is onto. Denote $k=h^{-1}$. By Lemma \ref{1/h} we know that $k\in L^2_{\Pi}$ and that the Riesz projection is bounded in $L^2_{|k|^2}(\R)$. The relation $\Theta\overline{I}= \overline{h}/{h}$ implies that $\overline{\Theta}I=\overline{k}/{k}$. We may therefore repeat the argument above to conclude that $P_{I}|K_{\Theta}$ is bounded from below. It follows that $P_{\Theta}|K_I=(P_{I}|K_{\Theta})^*$ is onto.
\end{proof}

\subsection{Norm estimates for the Riesz projection}\label{A_2 subsection}
In light of Proposition \ref{thm; exp rb modal} and Theorem \ref{rb as riesz}, results regarding norm estimates of the Riesz projection may be applied when studying lower Riesz basis bounds. The goal of this subsection is to shortly survey the literature concerning such estimates (and to note a slight improvement in one of them). The results below are formulated in terms of the Riesz projection. We remark that in the literature these results are commonly formulated in terms of the Hilbert transform. It is straightforward to verify that these operators have equivalent norms.

Weights $w$ for which the Riesz projection is bounded were characterized independently by Helson and Szeg\"{o} \cite{HS} and by Hunt, Muckenhoupt, and Wheeden \cite{HMW}. They present two different (albeit, as it turns out, equivalent) characterizations.

In \cite{HMW} Hunt, Muckenhoupt, and Wheeden prove that the Riesz projection is bounded if and only if $w\in A_2(\R)$, that is,
\begin{equation}\label{A_2 condition}
[w]_{A_2}:=\sup_{I=[a,b]\subseteq\R}\frac{1}{|I|^2}\int_Iw\int_I\frac{1}{w}<\infty.
\end{equation}
It was conjectured that there exists a universal constant $C>0$ so that the corresponding  operator norms satisfy
\begin{equation}\label{a2 bound}
\|R\|_w \leq C[w]_{A_2}.
\end{equation}
 This conjecture was confirmed by Petermichl, \cite{Peter}. The result is sharp in the sense that the linear power of $[w]_{A_2}$ is the best possible.

To present Helson and Szeg\"{o}'s characterization we first require some notations. We use the following convention for the (regularised version of the) Hilbert transform, which is well defined for $v\in L^1_{\mathbf{\Pi}}(\R)$
\[
\widetilde{v}(x)=\frac{1}{\pi}\textrm{P.V.}\int_{\R}\Big( \frac{1}{x-t}+\frac{t}{1+t^2}\Big)v(t) dt.
\]
In particular, with this convention the Hilbert transform is defined for $v\in L^{\infty}(\R)$. It is well known that for such $v$ we have $\widetilde{v}\in L^1_{\mathbf{\Pi}}(\R)$ and so $\widetilde{\tilde{v}}$ is defined and satisfies $\widetilde{\tilde{v}}=-v+const$. Denote
\[
\textrm{BMO}(\R):=\{u+ \widetilde{v}:\:u,v\in L^{\infty}(\R)\}.
\]
It follows that the Hilbert transform is defined on $\textrm{BMO}(\R)$ and maps it on itself.

%We note that the extension
%\begin{equation}\label{swartz}
%V(x)=\frac{i}{\pi}\int_{\R}\Big( \frac{1}{z-t}+\frac{t}{1+t^2}\Big)v(t) dt
%\end{equation}
%recovers the function $V$ from its real part $v$, provided $V(i)=0$ and $V\in H^{\infty}$. In particular this implies that the poisson extension of $\widetilde{v}$ is zero at %$i$, or in other words, that the collection $\widetilde{L^{\infty}}:=\{\widetilde{v}:v\in L^{\infty}\}$ does not contain a constant function.

In \cite{HS}, Helson and Szeg\"{o} prove that the Riesz projection is bounded in $L^2_w(\R)$ if and only if there exist $u,v\in L^{\infty}(\R)$ with $\|v\|_{\infty}< \pi/2$ so that
\begin{equation}\label{hs decomp}
w=e^{u+\widetilde{v}}.
\end{equation}
A corresponding estimate for the operator norms was recently obtained by Carro, Naibo, and Soria-Carro \cite{CaNaSoCa}. For $u\in L^{\infty}(\R)$ define
$\|u\|_{\infty}^0:=\inf_{c\in\R}\|u-c\|_{\infty}$ and note that if $u$ is real valued then $\sup u-\inf u=2\|u\|_{\infty}^0$. For $w$ as in (\ref{hs decomp}) Carro, Naibo, and Soria-Carro define
\[
[w]_{A_2(HS)}:=\inf_{u,v}\frac{e^{\|u\|_{\infty}^0}}{\cos \|v\|_{\infty}},
\]
where the infimum is taken over all $u,v\in L^{\infty}(\R)$ with $\|v\|_{\infty}< \pi/2$ for which (\ref{hs decomp}) holds. They show that
\begin{equation}\label{hs bound}
\|R\|_w \leq C[w]_{A_2(HS)},
\end{equation}
and that this estimate is sharp in the sense that there exist $w\in A_2(\R)$ for which the norms are equivalent. In \cite{CaNaSoCa} the estimate (\ref{hs bound}) is obtained via so called Rellich-type identities. It turns out that this estimate can be obtained also by a careful choice of constants in the Helson and Szeg\"{o} proof, and that the constant $C$ in (\ref{hs bound}) is redundant, that is,
\begin{equation}\label{hs bound -2}
\|R\|_w \leq [w]_{A_2(HS)}.
\end{equation}
We include this argument here. It is a direct consequence of the following three lemmas, the first of which asserts that the condition on $w$, imposed in Theorem \ref{rb as riesz}, was not restrictive. It can be found e.g. in \cite{HNP} (pp 66-67).

\begin{lemma}\label{a_2 poisson}
Let $w\in A_2(\R)$ then $w\in L^1_{\Pi}(\R)$.
\end{lemma}

The next lemma can be found e.g. in \cite{Nikbook} (combine Theorem 4.3.1(c) with the proof of Theorem 4.6.1(4)).

\begin{lemma}\label{nehari}
Let $w\in A_2(\R)$ and let $h$ be the outer function from Lemma \ref{1/h}, so that $w=|h|^2$. Then
\[
\|R\|^{-2}_w=1-\textrm{dist}^2_{L^{\infty}(\R)} \Big(\frac{\overline{h}}{h}, \mathbf{H}^{\infty}\Big),
\]
where  $\mathbf{H}^{\infty}$ denotes the $\infty$-norm Hardy space in the upper half plane.
\end{lemma}

Finally, we apply the following lemma to obtain (\ref{hs bound -2}). 

\begin{lemma}
Let $w$ and $h$ be as in Lemma \ref{nehari}. Then
\[
1-\textrm{dist}^2_{L^{\infty}(\R)} \Big(\frac{\overline{h}}{h}, \mathbf{H}^{\infty}\Big)\geq [w]^{-2}_{A_2(HS)}
\]
\end{lemma}
\begin{proof}
Put $\phi=\overline{h}/{h}=e^{-i\:\textrm{arg}(h^2)}$.
Condition (\ref{hs decomp}) implies that $\log |h|^2=u+\widetilde{v}$ and so $\widetilde{\log} |h|^2=\widetilde{u}-v+a$ for some constant $a$. Consequently
\[
\phi =e^{-i(\widetilde{u}-v+a)}.
\]
Fix $c\in\R$. For $r>0$ let
$
g_r=r e^{-(u+i\widetilde{u}+ia-c)}.
$
Then $g_r\in \mathbf{H}^{\infty}$ and on $\R$ we have
\begin{align*}
|\phi -g_r|^2&= |1 -\overline{\phi}g_r|=(1 -r e^{ -(u-c)}\cos{ v})^2+(r e^{- (u-c)}\sin{ v})^2.
\end{align*}
Hence, keeping in mind that $\|v\|_{\infty}<\pi/2$, we get
\begin{align*}
1-|\phi -g_r|^2&=2r e^{-(u-c)}\cos{v}-r^2 e^{-2(u-c)}\\&\geq r e^{-(u-c)} (2\cos{\|v\|_{\infty}}-r e^{-(u-c)}).
\end{align*}
Choosing $r=e^{-\|u-c\|_{\infty}}\cos\|v\|_{\infty}$, and noting that $e^{-\|u-c\|_{\infty}-(u-c)}<1$, completes the proof.
\end{proof}

The discussion above implies that both $[w]_{A_2}$ and $[w]_{A_2(HS)}$ provide bounds for the norm of the Riesz projection. It is not difficult to show that $[w]_{A_2}\leq [w]_{A_2(HS)}^2$ (see \cite{CaNaSoCa}, pp 2-3) but the precise relation between these quantities is not known.

\section{Lower bound estimates in some Riesz basis characterizations}

In this section we estimate the lower Riesz basis bound in some Riesz bases characterizations from \cite{Pav, HNP, SeipLiubAp}.

\subsection{Riesz bases bounds via generating functions}\label{charact-generat}

 The following are well known considerations: Let $\L\subseteq \R$ and assume that  $E(-\L)$ is a Riesz basis in $L^2[0,1]$. Let $\{g_{\lambda}\}$ be the unique dual system described in Section \ref{section abstract setting}. For $\lambda\in\L$, let $G_{\lambda}$ be the inverse Fourier transform of $g_{\lambda}$. Then $G_{\lambda}$ is equal zero on $\L\setminus \{\l\}$ and it does not have any other zero in $\mathbb{C}$. Indeed, if $w$ were an additional zero then by the Paley-Wiener theorem, the function
 \[
 G_{\lambda}(z)\frac{z-\l}{z-w}
 \]
  would have been the inverse Fourier transform of a function in $L^2[0,1]$ which is orthogonal to $E(-\Lambda)$, contradicting the completeness property of the system.

  In a similar way one can show that if $\l_1\in\L$ then
\[
G_{\lambda_1}(z)=c(\l_1,\l)G_{\lambda}(z)\frac{z-\l}{z-\l_1},
\]
for some constant $c(\l_1,\l)\in \mathbb{C}$. By the Paley-Wiener theorem, this implies that if $g_{\l}$ is supported on a subinterval of $[0,1]$ then the same is true for all $g_{\l_1}$, $\l_1\in\L$ (see e.g. \cite{Koosis} pp 132-133). As the identity (\ref{dual riesz}) holds for all $f\in L^2[0,1]$ this leads to a contradiction. We may therefore conclude that the functions $g_{\l}$ have full support on the interval $[0,1]$.

For $\l\in\L$ consider the function $F=(z-\l)G_{\l}$. Then $F^2\in L^2_{\Pi}(R)$ and the zero set of $F$ is $\L$. The Paley-Wiener theorem implies that $F$ is of exponential type $2\pi$ and moreover, since $g_{\l}$ has full support on the interval $[0,1]$, that
\begin{equation}\label{eq;full diagram}
\lim\sup_{t\rightarrow \infty}\frac{\log{|F(it)|}}{t}=0\qquad\textrm{and}\qquad
\lim\sup_{t\rightarrow \infty}\frac{\log{|F(-it)|}}{t}=2\pi.
\end{equation}
(See e.g. \cite{Koosis} pp 132-133). In particular, it follows that $F$ is in the Cartwright class and so the following product converges
\begin{equation}\label{eq;product}
\lim_{R\rightarrow \infty}\prod_{\l\in\L,\:|\lambda|<R}\big(1-\frac{z}{\lambda}\big),
\end{equation}
(where if $0\in\L$ then the corresponding term in the product is replaced by $z$). Moreover, up to a multiplicative constant $F$ is equal to
\[
F_{\L}(z):=e^{\pi i z}\lim_{R\rightarrow \infty}\prod_{\l\in\L,\:|\lambda|<R}\big(1-\frac{z}{\lambda}\big),
\]
(see e.g. \cite{HNP} pp 284-285).

We say that a uniformly discrete sequence $\L\subset\R$ admits an \textit{appropriate generator} if the product (\ref{eq;product}) converges and the corresponding function $F_{\L}$ is of exponential type $2\pi$, satisfies (\ref{eq;full diagram}), and $F_{\L}\in L^2_{\Pi}(\R)$. The discussion above implies that if $E(\L)$ is a Riesz basis in $L^2[0,1]$ then $\L$ admits an appropriate generator.

Next, assume that $\L\subseteq\R$ is uniformly discrete and that it admits an appropriate generator. Fix $y>0$ and denote
\begin{equation}\label{eq; F after translate}
F_{\L-iy}(z):=e^{\pi i z}\lim_{\R\rightarrow \infty}\prod_{\l\in\L,\:|\lambda|<R}\big(1-\frac{z}{\lambda-iy}\big).
\end{equation}
Since the product (\ref{eq;product}) converges the product defining $F_{(\L-iy)}$ converges as well and we have
\begin{equation}\label{eq; F-lambda identity}
F_{\L-iy}(z)F_{\L}(iy)=F_{\L}(z+iy).
\end{equation}
In particular, this identity implies that $F_{\L-iy}$ is of exponential type $2\pi$, satisfies (\ref{eq;full diagram}), and that $F_{\L-iy}\in L^2_{\Pi}(\R)$ (see e.g. \cite{Youngbook} p79). It follows that when restricted to the upper half plane $F_{\Lambda-iy}$ admits an inner-outer factorization. As it is analytic on the whole plane and has no zeros in the upper half plane, condition  (\ref{eq;full diagram}) implies that $F_{\Lambda-iy}$ is outer. 

Further, observe that if a uniformly discrete $\L$ admits an appropriate generator then, as the product in (\ref{eq; F after translate}) converges, the product
\[
\lim_{R\rightarrow\infty}\prod_{|\l|<R}\frac{1-\frac{z}{\l+iy}}{1-\frac{z}{\l-iy}}
\]
converges as well, and so is equal up to a unimodular multiplicative constant to 
$
B_{\L_y}(z)
$.
Consequently, under these conditions, there exists a unimodular constant $\varepsilon$ so that the outer function $h=\varepsilon F_{\Lambda-iy}$ satisfies
\begin{equation}\label{the h condition}
\overline{S}(x)B_{\Lambda_y}(x)= \frac{\overline{h(x)}}{h(x)} \quad x\in\R.
\end{equation}

 Combined with the discussion in the previous section this allows us to reproduce Pavlov's characterization of exponential Riesz bases and to provide an estimate for the corresponding lower Riesz basis bound.

\begin{theorem}\label{thm; HNP generating}
Let $\Lambda\subseteq\R$ be uniformly discrete. The exponential system $E(\L)$ is a Riesz basis in $L^2[0,1]$ if and only if  the product (\ref{eq;product}) converges and the corresponding function $F_{\L}$ is of exponential type $2\pi$, satisfies (\ref{eq;full diagram}), and for some $y>0$ we have
\begin{equation}\label{w is a2}
w_{\Lambda}^{(y)}(x):=|F_{\Lambda}(x+iy)|^2\in A_2(\R).
\end{equation}
Moreover, in this case (\ref{w is a2}) holds for all $y>0$ and the lower Riesz basis bound satisfies both estimates:
\begin{equation}\label{hs rb bound}
A(\Lambda)\geq \sup_y \frac{1}{7\delta(\Lambda)}e^{-{8\pi y}/{\delta(\Lambda)}}[w_{\Lambda}^{(y)}]^{-2}_{A_2(HS)}
\end{equation}
and
\begin{equation}\label{a2 rb bound}
A(\Lambda)\geq \sup_y \frac{C}{\delta(\Lambda)}e^{-{8\pi y}/{\delta(\Lambda)}}[w_{\Lambda}^{(y)}]^{-2}_{A_2},
\end{equation}
where $C>0$ is a universal constant.
\end{theorem}

\begin{proof}
Assume first that $E(\L)$ is a Riesz basis. We have already seen that in this case $\L$ admits an appropriate generator and that the outer function $h=\varepsilon F_{\Lambda-iy}$ satisfies (\ref{the h condition}). The identity (\ref{eq; F-lambda identity}) implies that $w_{\Lambda}^{(y)}=C|h|^2$ for some constant $C>0$.  Proposition  \ref{thm; exp rb and projections}', Theorem \ref{rb as riesz}, and the Hunt, Muckenhoupt, and Wheeden characterization, now ensure that $w_{\Lambda}^{(y)}(x)\in A_2(\R)$. The Riesz basis bounds follow by inserting (\ref{a2 bound}) and (\ref{hs bound -2}) into Proposition \ref{thm; exp rb and projections}'.

In the opposite direction, assume that $F_{\Lambda}$ as required exists. Lemma \ref{a_2 poisson} ensures that $w_{\Lambda}^{(y)}\in L^1_{\Pi}(\R)$. Similar considerations as above imply that there exists an outer function $h$ so that (\ref{the h condition}) holds and $w_{\Lambda}^{(y)}=C|h|^2$. Proposition  \ref{thm; exp rb and projections}', Theorem \ref{rb as riesz}, and the Hunt, Muckenhoupt, and Wheeden characterization, may now be applied to conclude that $E(\Lambda)$ is a Riesz basis in $L^2[0,1]$.

\end{proof}

Let $\L\subseteq\R$ be uniformly discrete and assume that the product (\ref{eq;product}) converges to some function $H_{\L}(z)=e^{-\pi i z}F_{\L}(z)$. Denote
\[
G_{\Lambda}(x):=\frac{H_{\L}(x)}{\textrm{dist}(x,\Lambda)}\qquad x\in\R.
\]
Seip and Lyubarskii observed in \cite{SeipLiubAp} that for $y>0$ the functions $|H_{\Lambda}(x+iy)|$ and $|G_{\Lambda}(x)|$ are equivalent, and so the latter may replace the former in Theorem \ref{thm; HNP generating}. The following lemma gives an explicit estimate for this equivalence (see also \cite{HNP}, lemma 1.5, p285). 

\begin{lemma}\label{s like p}
Let $\L\subseteq\R$ and $H_{\Lambda}$ be as above. Then for every $y>0$ there exists a function $\nu_y\in L^{\infty}$ with
\[
\|\nu_y\|_{\infty}\leq \frac{2\pi y(\delta^2(\L)+4y^2)}{\delta^3(\L)}.
\]
so that
\[
\frac{|H_{\Lambda}(x)|^2}{|H_{\Lambda}(x+iy)|^2}=e^{\nu_y}\frac{\textrm{dist}^2(x,\Lambda)}{\textrm{dist}^2(x+iy,\Lambda)}.
\]
\end{lemma}
\begin{proof}
Put $\delta=\delta(\L)$ and let $\tau=\delta/2$. Fix $x\in\R$ and let $\l_0\in\Lambda$ be a point satisfying $|x-\l_0|=\textrm{dist}(x,\Lambda)$. Note that $\textrm{dist}(x,\Lambda\setminus\{\l_0\})\geq\tau$. We have
\[
\begin{aligned}
\frac{|H_{\Lambda}(x)|^2}{|H_{\Lambda}(x+iy)|^2}&=\prod\frac{|x-\lambda|^2}{|x-\lambda|^2+y^2}
=\frac{\textrm{dist}^2(x,\Lambda)}{\textrm{dist}^2(x+iy,\Lambda)}\prod_{\lambda\neq\lambda_0}\Big(1-\frac{y^2}{|\lambda-x|^2+y^2}\Big).
\end{aligned}
\]

Denote
\[
\nu_y(x):=\sum_{\lambda\neq\lambda_0}\log{\Big(1-\frac{y^2}{|\lambda-x|^2+y^2}\Big)}.
\]
Note that if $0<t\leq r<1$ then $-t/(1-r)<\log(1-t)$. Putting $r=y^2/(\tau^2+y^2)$ (and recalling that if $\lambda\neq \lambda_0$ then $|\lambda-x|\geq \tau$) we get
\[
\begin{aligned}
0&\geq \nu_y(x)\geq -\big(1+\frac{y^2}{\tau^2}\big)\sum_{\lambda\neq\lambda_0}\frac{y^2}{|\lambda-x|^2+y^2}.
\end{aligned}
\]
Since $\Lambda$ is separated by $\delta=2\tau$ and $|\lambda-x|\geq \tau$ for $\lambda\neq \lambda_0$, we have
\[
\begin{aligned}
\sum_{\lambda\neq\lambda_0}\frac{y^2}{|\lambda-x|^2+y^2}&\leq 2\sum_{k=1}\frac{y^2}{(\tau k)^2+y^2}\leq \frac{y}{\tau}\int_{\R}\frac{dx}{x^2+1}= \frac{\pi y}{\tau}\\
\end{aligned}
\]
\end{proof}

Choosing $y=\delta(\Lambda)$ in Theorem \ref{thm; HNP generating} and Lemma \ref{s like p} we obtain the following version of  Theorem \ref{thm; HNP generating}.

\begin{thmmbis}{thm; HNP generating} \label{thm seip lyub}
Let $\Lambda\subseteq\R$ be uniformly discrete. The exponential system $E(\L)$ is a Riesz basis in $L^2[0,1]$ if and only if the product (\ref{eq;product}) converges to a function with of exponential type $2\pi$, condition (\ref{eq;full diagram}) holds, and $|G_{\Lambda}|^2\in A_2(\R)$.
Moreover, in this case the lower Riesz basis bound satisfies both estimates:
\[
A(\Lambda)\geq  e^{-20\pi}\frac{\delta(\Lambda)}{m^2(\Lambda)}[|G_{\Lambda}|^2]^{-2}_{A_2(HS)},\quad\textrm{and}\quad A(\Lambda)\geq  C\frac{\delta(\Lambda)}{m^2(\Lambda)}[|G_{\Lambda}|^2]^{-2}_{A_2},
\]
where $m(\Lambda)$ denotes the maximal distance between two consecutive elements in $\Lambda$ and $C>0$ is a universal constant.
\end{thmmbis}

\begin{remark}\label{dif in bmo}
Lemma \ref{s like p} implies in particular that
\[
\log{|F_{\Lambda}(x)|^2}-\log{|F_{\Lambda}(x+iy)|^2}\in \textrm{BMO}(\R).
\]
Indeed, on an interval of the form $(\l_n-\tau, \l_n+\tau)$ we have
\[
\frac{\textrm{dist}^2(x,\Lambda)}{\textrm{dist}^2(x+iy,\Lambda)}=\frac{|x-\l_n|^2}{|x-\l_n|^2+y^2}
\]
while outside the union of all such intervals
\[
\frac{\tau^2}{\tau^2+y^2}\leq \frac{\textrm{dist}^2(x,\Lambda)}{\textrm{dist}^2(x+iy,\Lambda)}\leq 1.
\]
As $\log |x|\in \textrm{BMO}(\R)$ the conclusion follows (this is easiest seen via the equivalent characterization for BMO, see \cite{Garbook} pp 222-223, we omit the details).
\end{remark}

\subsection{Riesz bases bounds via phase functions}

For $w\in\mathbb{C}\setminus\R$ we have
\[
\log{(1-\frac{x}{w})}= \int_0^x\frac{1}{t-w}\:dt.
\]
By taking the imaginary part we find that
\begin{equation}\label{arg}
\textrm{arg}\Big(1-\frac{x}{\lambda\pm iy}\Big)=\pm \int_0^x\frac{y}{(t-\lambda)^2+y^2}\: dt.
\end{equation}

Given a uniformly discrete sequence $\L\subseteq\R$ and $y>0$ denote
\[
\alpha_{\L_y}(x)=2\int_0^x\sum\frac{y}{(t-\lambda)^2+y^2}\: dt -2\pi x.
\]
The relation (\ref{arg}) implies that up to an additive constant, $\alpha_{\L_y}$ is a continuous branch of the argument of $\overline{S}B_{\L_y}$ over $\R$, and that if $\L$ admits an appropriate generator then $-\alpha_{\L_y}/2$ is a similar branch of the argument of $F_{\Lambda-iy}$.

With this we can reproduce the characterization of exponential Riesz bases via the phase function formulated in \cite{HNP}, and provide an estimate for the corresponding lower Riesz basis bound.

%Recall the notation $w_{\Lambda}^{(y)}=|F_{\Lambda}(x+iy)|^2$ from the formulation of Theorem \ref{thm; HNP generating}.

\begin{theorem}\label{thm; exp rb and phase}
Let $\Lambda\subseteq\R$ be a uniformly discrete sequence. Then $E(\Lambda)$ is a Riesz basis in $L^2[0,1]$ if and only if there exist, $y>0$, two bounded functions $v_y,u_y\in L^{\infty}(\R)$ with $\|v_y\|_{\infty}<\pi/2$, and a constant $c\in\mathbb{\R}$ such that
\[
\alpha_{\L_y}=v_y+\widetilde{u}_y+c.
\]
Moreover, in this case such a decomposition exists for all $y>0$, and the lower Riesz basis bound satisfies
\begin{equation}\label{eq: riesz bound estimate phase}
A(\Lambda)\geq \sup_y \frac{1}{7\delta(\Lambda)}e^{-(8\pi \frac{y}{\delta(\Lambda)}+  \|u_y\|^0_{\infty})}\cos^2{\|v_y\|_{\infty}}.
\end{equation}
\end{theorem}
\begin{proof}
Assume first that $E(\L)$ is a Riesz basis in $L^2[0,1]$ and fix $y>0$.  Theorem \ref{thm; HNP generating} and the HS characterization imply that there exist $v_y,u_y\in L^{\infty}(\R)$ with $\|v_y\|_{\infty}<\pi/2$ so that
\[
\log{ |F_{\Lambda }(x+iy)|^2} = -u_y(x)+\widetilde{v}_y(x).
\]
The identity (\ref{eq; F-lambda identity}) now implies that
\[
\log{ |F_{\Lambda -iy }(x)|^2} = -u_y(x)+\widetilde{v}_y(x) +\textrm{Const}.
\]
We have seen that up to an additive constant the argument of $F_{\Lambda -iy}$ is equal to $-\alpha_{\Lambda_y}/2$, and so we have
\begin{equation}\label{logF is elpha}
\widetilde{\log}{ |F_{\Lambda -iy}(x)|^2} = -\alpha_{\Lambda}(x)+\textrm{Const}.
\end{equation}
Combined, the last two identities give the required decomposition.

In the opposite direction assume that $\alpha_{\L_y}=v_y+\widetilde{u}_y+c$ and recall that up to an additive constant $\alpha_{\L_y}$ is a continuous branch of the argument of $\overline{S}B_{\L_y}$ over $\R$. By  Lemma \ref{a_2 poisson} we have $e^{u_y-\widetilde{v}_y+c}\in L^1_{\Pi}(\R)$ and so there exists an outer function $h_1$ which satisfies
\begin{equation}\label{who's h}
{\log}{|h_1|^2}=\widetilde{v}_y-u_y.
\end{equation}
This relation implies that, up to an additive constant, the argument of $h_1$ is equal to  $-\alpha_{\L_y}/2$. Let $h=\epsilon h_1$, where $\epsilon$ is a unimodular constant chosen so that on $\R$ we have $\overline{S}B_{\L_y}=\overline{h}/h$. Observe that $h\in L^2_{\Pi}(\R)$. Proposition  \ref{thm; exp rb and projections}', Theorem \ref{rb as riesz} and the Helson Szego characterization, now ensure that $E(\Lambda)$ is a Riesz basis in $L^2[0,1]$. The estimate for the Riesz basis bound follows from (\ref{hs bound -2}).

 \end{proof}

\subsection{Riesz bases bounds via counting functions}

Let $\L\subseteq\R$ be uniformly discrete. The counting function of $\L$ is defined by
\[
N_{\Lambda}(x):=\left\{\begin{array}{cc}
|\L\cap [0,x]|,& x\geq 0\\
-|\L\cap [x,0)|,& x< 0. \end{array}\right..
\]

If the generating function $F_{\Lambda}$ is defined then there exists an intimate relation between it and the counting function. This is formulated in the next lemma, which can be found in \cite{HNP}. We add a proof for completeness.
\begin{lemma}\label{lemma; N as boundry limit of F}
Let $\L\subseteq\R$ and assume that it admits an appropriate generator, in the sense defined in Section \ref{charact-generat}. Then, when restricted to the upper half plane $F_{\Lambda}$ is an outer function and
\[
\lim_{y\rightarrow 0}\textrm{Im}(\log{F_{\Lambda}^2(x+iy)})=2\pi x-N_{\Lambda}(x) +\textrm{Const},
\]
\end{lemma}
\begin{proof}
As it is an appropriate generator, $F_{\L}$ admits an inner-outer factorization in the upper half plane. Since it is analytic on the whole plane and it has no zeroes in the upper half plane, condition  (\ref{eq;full diagram}) implies that $F_{\Lambda}$ is outer. We therefore have
\[
\log{F_{\Lambda}^2(z)}=2\pi i z+ 2\sum_n\log(1-\frac{z}{\lambda_n})\quad z\in \mathbb{C}^+.
\]
Observe that
\[
\lim_{y\rightarrow 0}\log(1-\frac{x+iy}{\lambda_n})=\log\big|1-\frac{x}{\lambda_n}\big|-\pi i \left\{\begin{array}{cc}
1\!\!1_{[\lambda_n,\infty)},& \lambda_n\geq 0\\
-1\!\!1_{(-\infty, \lambda_n]},& \lambda_n< 0. \end{array}\right..
\]
The statement in the lemma follows.
\end{proof}

A similar intimate relation exists between the counting function and phase function. This essentially follows from the discussion in the previous section and will be made more clear throughout the proof of Theorem \ref{thm; HNP generating and counting}. To formulate it explicitly we will need the following notation. For a Borrel measure $\mu$ which satisfies $\int_{\R} d\mu/(1+x^2)<\infty$ we denote the Poisson extension of $\mu$ by
\[
P_y[\mu](x):=\int_R\frac{y}{(x-t)^2+y^2}d\mu(t).
\]
If $d\mu = fdt$ for some $f\in L^1_{\Pi}(\R)$ we denote for brevity $P_y[f]$. The following observation follows from considering the poisson extension as a convolution and performing an appropriate integration by parts, we omit the details.

\begin{lemma}\label{lemma; Py using measure}
Let $\mu$ be a Borrel measure which satisfies $\int_{\R} d\mu/(1+x^2)<\infty$, and let $f$ be defined by
\begin{equation}\label{antiderivative}
f(x)=\int_0^xd\mu.
\end{equation}
If $f\in L^1_{\Pi}(\R)$ then for every $y>0$ there exists a constant $C(y)$ so that
\begin{equation}\label{poisson identity}
P_y[f](x)=\int_0^xP_y[\mu](\tau)d\tau+ C(y).
\end{equation}
\end{lemma}

With the above two observations established, we can reproduce the characterization of exponential Riesz bases via the counting function formulated in \cite{HNP}, and provide an estimate for the corresponding lower Riesz basis bound. We note that the proof for the first direction below is similar to the proof given in \cite{HNP}, and is included here for completeness.

\begin{theorem}\label{thm; HNP generating and counting}
Let $\Lambda\subseteq\R$ be uniformly discrete and denote $\psi(x)= 2\pi x-N_{\Lambda}(x)$. Then the exponential system $E(\L)$ is a Riesz basis in $L^2[0,1]$ if and only if  $\psi\in \textrm{BMO}(\R)$, and there exist $y>0$, two bounded functions $v_y,u_y\in L^{\infty}$ with $\|v_y\|_{\infty}<\pi/2$, and a constant $c\in\mathbb{\R}$ so that
\begin{equation}\label{HS for N}
P_y[\psi](x)=v_y(x)+\widetilde{u}_y(x)+c.
\end{equation}
Moreover, in this case such a decomposition holds for all $y>0$ and the lower Riesz basis bound of $E(\L)$ satisfies (\ref{eq: riesz bound estimate phase}).
\end{theorem}

\begin{proof}
Assume first that $E(\L)$ is a Riesz basis in $L^2[0,1]$ and fix $y>0$.  Theorem \ref{thm; HNP generating} implies that there exist $v_y,u_y\in L^{\infty}$ with $\|v_y\|_{\infty}<\pi/2$ so that
\begin{equation}\label{again with logF}
\log{ |F_{\Lambda }(x+iy)|^2} = -\widetilde{v}_y(x)+{u}_y(x).
\end{equation}
In particular, this implies that $\log{|F_{\Lambda }(x+iy)|^2}\in \textrm{BMO}(\R)$ and so, due to Remark (\ref{dif in bmo}), the same is true for $\log{|F_{\Lambda }(x)|^2}$. It follows that the Hilbert transform of $\log{|F_{\Lambda }(x)|^2}$ is defined and by Lemma \ref{lemma; N as boundry limit of F}, that it is equal up to a constant to
$\psi$. We conclude that $\psi\in \textrm{BMO}(\R)$ and that its Poisson extension is equal to $\widetilde{\log}{ |F_{\Lambda }(x+iy)|^2}$. The conclusion now follows from (\ref{again with logF}).

In the opposite direction, assume that $\psi\in \textrm{BMO}(\R)$ and that (\ref{HS for N}) holds. Consider the measure $\mu$ given by $d\mu=2\pi dx-\sum_{\lambda\in\Lambda}\delta_{\lambda}$, where $\delta_{\lambda}$ is the standard $\delta$ measure translated by $\lambda$. Observe that the function $\psi$ and the measure $\mu$ are related by (\ref{antiderivative}). Since $\L$ is uniformly discrete we have $\int_{\R} d\mu/(1+x^2)<\infty$ and moreover, since we assume that $\psi\in \textrm{BMO}(\R)$ we also have $\psi\in L^1_{\Pi}(\R)$.  Lemma \ref{lemma; Py using measure} may now be applied and the identity (\ref{poisson identity}) holds. Applying  (\ref{HS for N}) we may conclude that there exist $y>0$, two bounded functions $v_y,u_y\in L^{\infty}$ with $\|v_y\|_{\infty}<1/4$, and a constant $c(y)\in\R$ so that
\[
\int_0^xP_y[\mu](\tau)d\tau=v_y(x)+\widetilde{u}_y(x)+c(y).
\]
One may check that the integral on the left hand side is equal to $-\alpha_{\L_y}(x)$. The conclusion now follows by applying Theorem \ref{thm; exp rb and phase}.

\end{proof}
\begin{remark}\label{relation between things}
Assume that $E(\L)$ is a Riesz basis in $L^2[0,1]$ and let $y>0$. It is perhaps worth noting that under this condition, the following relations between the generating, phase, and counting functions have been established:
\[
\begin{aligned}
&\widetilde{\log}{|F_{\L}(x)|^2}=2\pi x-N_{\Lambda}(x)+\textrm{Const},\\
&\widetilde{\log}{|F_{\L}(x+iy)|^2}=P_y(2\pi x-N_{\Lambda}(x))+\textrm{Const}= -\alpha_{\L_y}(x)+\textrm{Const}.
\end{aligned}
\]
\end{remark}

\section{Riesz basis bounds in some classical results}

In this section we study the lower Riesz basis bound of certain families of exponential Riesz bases. We consider Avdonin's theorem, Levin and Golovin's theorem, and applications to Marcinkiewicz-Zygmund families. Some additional examples are given. The estimates we provide are in general not sharp, as we discuss below.

\subsection{Perturbations of the integers}\label{kadetz lindner avdonin}
 For $\Lambda=\Z$ the product (\ref{eq;product}) is equal to $(\sin\pi z)/\pi$ and so, up to a multiplicative constant, $F_{\Lambda}(z)=e^{2\pi i z}-1$. For a fixed $y>0$ Remark \ref{relation between things} now implies that
 \begin{equation}\label{alpha Z}
\alpha_{\Z_y}(x)=-\widetilde{\log}{|e^{2\pi(-y+ i x)}-1|^2} +\textrm{Const}.
\end{equation}
With this one can estimate the lower Riesz basis bound in cases where $\L$ is not "too far" from $\Z$.

\begin{corollary}\label{cor; general stability}
Let $\Lambda\subseteq\R$ be uniformly discrete and assume that $\mu_n=\lambda_n-n$ satisfy
$
\sup_n{|\mu_n|}<\infty.
$
For $y>0$ denote
\[
\tau_y:=\sup_{x\in{\R}}\Big|{\sum_n\int_{x}^{x+\mu_n}\frac{y}{(n -t)^2+y^2}\: dt}\big|.
\]
If there exists $y_0\geq 1$ so that $\tau_{y_0}<\pi/4$ then $E(\Lambda)$ is a Riesz basis in $L^2[0,1]$ and its lower Riesz basis bound satisfies
%\[
%A(\Lambda)\geq \frac{3}{\pi}\frac{e^{-8\pi (\frac{y}{\delta})}}{1- e^{-y}}\cos^2{2\pi \tau_y}\quad \forall y\geq y_0.
%\]
%In particular this implies that
\begin{equation}\label{general bound estimate}
A(\Lambda)\geq \frac{1}{28\delta(\Lambda)}e^{-\frac{8\pi y}{\delta(\Lambda)}}\cos^2{2\tau_{y_0}}.
\end{equation}
\end{corollary}

\begin{proof}
Fix $y>0$. We first show that there exists a constant $c(y)$ so that
\begin{equation}\label{eq; distance estimate}
2\tau_y\geq \|\alpha_{\Z_y}-\alpha_{\Lambda_y}+c(y)\|_{\infty}.
\end{equation}
Indeed, this follows from
\[
\begin{aligned}
\frac{1}{2}\big(\alpha_{\Z_y}(x)-\alpha_{\L_y}(x)\big)&=\sum_n\int_{0}^{x}\frac{y}{(n -t)^2+y^2}\: dt-\sum_n\int_{0}^{x}\frac{y}{(n+\mu_n -t)^2+y^2}\: dt\\
&=\sum_n\int_{x-\mu_n}^{x}\frac{y}{(n -t)^2+y^2}\: dt-\sum_n\int_{-\mu_n}^{0}\frac{y}{(n -t)^2+y^2}\: dt.
\end{aligned}
\]
The sums above converge absolutely since the sequence $\{\mu_n\}$ is bounded. In particular, this implies that
the right most sum in the last expression converges to a constant which does not depend on $x$. The relation in  (\ref{eq; distance estimate}) follows.

Denote $v_y=\alpha_{\L+iy}(x)-\alpha_{\Z+iy}(x)-c(y)$. The relation (\ref{alpha Z}) implies that
\[
 \alpha_{\Lambda+yi}(x)=-\widetilde{\log}{|e^{2\pi(-y+ i x)}-1|^2}+v_y+\textrm{Const}.
\]
Note that if (for example) $y_0\geq 1$ then $1/2\leq |e^{2\pi(-y_0+ i x)}-1|\leq 2$. The Riesz basis bound estimate now follows from (\ref{eq: riesz bound estimate phase}) and (\ref{eq; distance estimate}).
\end{proof}

The expression $e^{ \frac{-8\pi y}{\delta}}$ in the bound estimate (\ref{general bound estimate}) is a 'price' we pay for shifting $\L$ by $iy$ in Section \ref{shift by y}. In general, it generates bound estimates which are not sharp. One way to demonstrate this defect is by considering the Kadetz $1/4$ theorem, mentioned in the introduction. It states that if $\mu_n=\lambda_n-n$ satisfy
\begin{equation}\label{kadec cond}
\mu:=\sup_n{|\mu_n|}<1/4
\end{equation}
 then $E(\L)$ is a Riesz basis in $L^2[0,1]$, and that an estimate for the corresponding lower Riesz basis bound is given by:
 \begin{equation}\label{kadec bound}
A(\L)\geq 2\sin ^2{\frac{\pi}{4}(1-4 \mu)}.
\end{equation}

It was observed by Hrusch\"{e}v \cite{Hrusch} that by comparing $\alpha_{\L_y}$ to $\alpha_{\Z_y}$, in much the same way as was done above, one can reproduce  this theorem. Indeed, by Poisson's summation formula we have
\begin{equation}\label{poission}
\sum_n\frac{y}{(n -t)^2+y^2}=\pi \frac{1-e^{-4\pi y}}{|1-e^{2\pi(-y+ it)}|^2}\leq \pi \frac{1+e^{-2\pi y}}{1-e^{-2\pi y}}.
\end{equation}
This expression converges uniformly to $\pi$ as $y$ tends to infinity. Condition (\ref{kadec cond}) now implies that $\tau_{y_0}<\pi/4$  for $y_0$ large enough, and the requirements of Corollary \ref{cor; general stability} are met. A corresponding estimate for the lower  Riesz basis bound can be found by noting that if $e^{-2\pi y_0}=(1-4\mu)/4<(1-4\mu)/(1+12\mu)$ then $\tau_{y_0}\leq \pi(1+4\mu)/8$. Since $\delta(\Lambda)\geq 1/2$, the estimate (\ref{general bound estimate}) gives the bound
\[
A(\Lambda)\geq \frac{1}{28\cdot 4^8}(1-4\mu)^8\sin ^2{\frac{\pi}{4}(1-4 \mu)}.
\]
This estimate is clearly worse then the estimate in (\ref{kadec bound}).

Nevertheless, Corollary \ref{cor; general stability} allows one to estimate the Riesz bases bound of general exponential Riesz bases, and in some cases to improve the dependence on certain parameters compared to estimates obtained in other methods.

To demonstrate this, consider Avdonin's averaged version of Kadetz' theorem, which was also formulated in the introduction. It was observed by  Hrusch\"{e}v that this theorem too can be reproduced via Pavlov's characterization  \cite{Hrusch} . Recall the formulation of the theorem: Let $\L\subseteq\R$ be a uniformly discrete sequence with separation constant $\delta:=\delta(\Lambda)$, and assume that there exists $L>0$ so that $\mu_n=\lambda_n-n$ satisfy
$
\sup_n{|\mu_n|}\leq L.
$  Avdonin's theorem states that if for some $N\in \N$ the averaged perturbations are well bounded, that is,
\begin{equation}\label{d_n}
\mu^*(N):=\sup_{m\in\Z}\Big|\frac{1}{N}\sum_{n=mN}^{(m+1)N-1}\mu_n\Big|<\frac{1}{4},
\end{equation}
then $E(\L)$ is a Riesz basis in $L^2[0,1]$. An estimate for the lower Riesz basis bound, with the parameters $\delta=\delta(\Lambda), L, N$ and $\theta:=1/4-\mu^*(N)$, was obtained by Lindner who (essentially) showed that
\[
A(\L)\geq e^{-20 \pi ^2(12(L+1))^{N+\frac{32L^2}{\theta}}}\big(\frac{\theta\delta}{108(L+1)}\big)^{240(12L)^{N+\frac{32L}{\theta}}}.
\]
Note that some parts of the last expression were modified for ease of presentation, though the rate of decay, which is double exponential in all parameters but $\delta$, was kept. See \cite{Lindner} for a precise formulation.

 The estimate (\ref{general bound estimate}) allows us to improve (and simplify) this bound, by reducing the rate of decay from a double exponential rate to an exponential one.
\begin{theorem}\label{thm:avdonin}
Let $\L\subseteq\R$ be uniformly discrete. Assume that there exists $L>0$ so that $\mu_n=\lambda_n-n$ satisfy
$
\sup_n{|\mu_n|}\leq L.
$  If for some $N\in\N$ the inequality (\ref{d_n}) holds then $E(\L)$ is a Riesz basis in $L^2[0,1]$ with lower Riesz basis bound
\[
A(\L)\geq \frac{1}{7\delta(\Lambda)}e^{- \frac{960\pi L^2N}{\delta(\Lambda)(1-4\mu^*)^2}}\sin^2{\frac{\pi}{4}(1-4\mu^*)},
\]
where $\mu^*:=\mu^*(N)$.
\end{theorem}
\begin{proof}
We follow Hrusch\"{e}v's proof, as presented in \cite{HNP}. We wish to apply Corollary \ref{cor; general stability}. For this we first note that  if
\[
f_n(x)=\frac{y}{(x-n)^2+y^2},
\]
 then by the mean value theorem there exists a point $\xi_n$ with $|x-\xi_n|\leq |\mu_n|\leq L$ so that
\[
\int_x^{x+\mu_n}f_n(t)\: dt=\mu_nf_n(\xi_n)=\mu_nf_n(x)\Big(1+\big(\frac{f_n(\xi_n)}{f_n(x)}-1\big)\Big)
\]
It follows that
\[
|\sum_n\int_x^{x+\mu_n}f_n(t)\:dt|\leq |\sum_n\mu_nf_n(x)|+\sum_n|\mu_n f_n(x)|\Big|\frac{f_n(\xi_n)}{f_n(x)}-1 \Big|=:|\textrm{I}|+\textrm{II}.
\]
To estimate II note that if $y\geq L\geq 1$ then
\[
\Big|\frac{f_n(\xi_n)}{f_n(x)}-1\Big| \leq |x-\xi_n|\frac{2|\xi_n-n|+|x-\xi_n|}{(\xi_n-n)^2+y^2}\leq L\big(\frac{1}{y}+\frac{L}{y^2}\big)\leq \frac{2L}{y}.
\]
Applying (\ref{poission}) we find that
\[
\textrm{II}\leq \frac{4\pi L^2}{y}.
\]

We turn to estimate $|I|$. For $x\in \R$ and $R>0$ denote
\[
\Delta_x(R)=\sum_{x-R\leq n\leq x+R}\mu_n.
\]
Let $r_0>1$ and put $R_0=r_0N$. We first note that if $R\geq R_0$ then
\begin{equation}\label{avdonin modified}
\frac{|\Delta_x(R)|}{2R} \leq\mu^*+\frac{L}{2r_0-1}.
\end{equation}
Indeed, to see this Let $k\in N$ be such that $kN\leq 2R\leq (k+1)N$.  Due to (\ref{d_n}), in each interval of the form $[mN,(m+1)N)$ the sum of the perturbations is bounded by $N\mu^*$.  Since $|\mu_n|\leq L$, the sum of the remaining perturbations is bounded by $LN$. Consequently $|\Delta_x(R)|\leq kN\mu^*+LN$ and so $|\Delta_x(R)|/2R\leq\mu^*+L/k$. As $(k+1)\geq 2r_0$ we obtain (\ref{avdonin modified}).

Observe that 
\[
\begin{aligned}
I&=\int_0^{\infty}\frac{y}{y^2+t^2}d\Delta_x(t)= \int_0^{\infty}\frac{\Delta_x(t)}{t}\frac{2yt^2}{(t^2+y^2)^2}dt\\
&=\Big(\int_0^{y^{-1}R_0}+\int_{y^{-1}R_0}^{\infty}\Big)\frac{\Delta_x(yt)}{2yt}\frac{4t^2}{(t^2+1)^2}dt=I_1+I_2,
\end{aligned}
\]
where $R_0=r_0N$ is as above.

The estimate (\ref{avdonin modified}) implies that
\[
I_2\leq (\mu^*+\frac{L}{2r_0-1})\int_0^{\infty}\frac{4t^2}{(t^2+1)^2}dt=(\mu^*+\frac{L}{2r_0-1})\pi.
\]
To estimate $I_1$ we note that for $0<t\leq R_0$ we have $\Delta_x(t)\leq 2R_0L$ and so
\[
I_1\leq \frac{2R_0L}{y}\int_0^{\infty}\frac{2t}{(t^2+1)^2}\: dt =\frac{2R_0L}{y}.
\]

Recall the notation $\tau_y$ from Corollary \ref{cor; general stability}. Combining the last few estimates we find that if $y\geq L$ then
\[
\tau_y\leq \textrm{I}_1+\textrm{I}_2+\textrm{II}\leq (\mu^*+\frac{L}{2r_0-1})\pi +\frac{2r_0 N L}{y}+\frac{4\pi L^2}{y}.
\]
Put
\[
r_0=\frac{8L}{1-4\mu^*}+1\quad\textrm{and }\quad y= \frac{120 L^2N}{(1-4\mu^*)^2}.
\]
With this choice, and since we may assume that $N\geq 2$, we obtain $\tau_y\leq \pi(1+4\mu^*)/8$. Plugging these in (\ref{general bound estimate}), the result follows.
\end{proof}

Theorem \ref{thm:avdonin} allows us to obtain uniform bounds in results which make use of Avdonin's theorem. We give one such example.

\begin{example}
In \cite{Seipint}, among other results, Seip shows that every subinterval of $[0,1]$ admits a Riesz basis of exponentials with integer frequencies. The proof makes use of Avdonin's theorem. Theorem \ref{thm:avdonin} allows us to strengthen this result as follows: Given any $0\leq a <1$ there exists a constant $C=C(a)$, so that any interval $I\subseteq [0,1]$ with $|I|\leq a$ admits a  Riesz basis $E(\Lambda)$  with $\Lambda\subseteq \Z$ and lower bound satisfying
\[
A(\L)\geq C(a)|I|.
\]
The dependence on $|I|$ in the above bound is optimal, as the norm of each element in the system is equal $\sqrt{|I|}$. It is well known that when $a=1$ the above result fails. This can be deduced, for example, from Theorem 4.17 in \cite{OlUlbook}. (This example follows an old discussion of the second author with Gady Kozma).
\end{example}

\subsection{Sine-type functions}\label{section: sine type}

Recall the family of weights $A_2(\R)$ which was defined in Section \ref{A_2 subsection}. Among the members of this family, a special role is played by those functions who not only satisfy (\ref{A_2 condition}), but are in fact bounded from above and below. Equivalently, these are the weights for which the decomposition (\ref{hs decomp}) holds with $v=0$. Weights in this subclass tend to provide results which are 'more similar' to those provided by the constant weight. 

For example, the system $E(\Z)$ in is an orthogonal basis in $L^2_w[0,1]$ if and only if $w$ is a constant, whereas it is a Schauder basis in the space if and only if $w\in A_2[0,1]$. It is not difficult to verify, that the weights $w$ for which $E(\L)$ is not only a Schauder basis, but in fact is a Riesz basis, are precisely those weights which are bounded from above and below.

Consider the family of generating functions characterized by Theorem \ref{thm; HNP generating}, that is, the family of appropriate generators  $F_{\Lambda}$  for which there exists $y>0$ so that
\[
w_{\Lambda}^{(y)}(x):=|F_{\Lambda}(x+iy)|^2\in A_2(\R).
\]
The subclass of this family for which there exists $y>0$ so that
\begin{equation}\label{bounded from above and below}
m\leq w_{\Lambda}^{(y)}(x)\leq M,
\end{equation}
where $m$ and $M$ are positive constants, is precisely the class of \textit{sine-type functions with real zeroes and with width of diagram equal to $2\pi$} which was introduced in the introduction. It was proved by Levin and Golovin that zero sets of functions from this class form Riesz bases in $L^2[0,1]$, \cite{Levin, Golovin}. Pavlov's characterization is an extension of this result. Theorem \ref{thm; HNP generating} allows us to give a simple estimate of the lower Riesz basis bound in the Levin-Golovin theorem.
\begin{corollary}\label{sine type}
Let $\Lambda\subseteq \R$ be uniformly discrete. If $\L$ is the zero set of a sine-type function $F_{\L}$ with width of diagram equal to $2\pi$, then $E(\L)$ is a Riesz basis in $L^2[0,1]$ with lower Riesz basis bound satisfying
\begin{equation}\label{sine type estimate}
A(\Lambda)\geq \frac{1}{7\delta(\L)}\frac{m}{M}e^{-\frac{8\pi y}{\delta(\L)}},
\end{equation}
where $y,m,M>0$ are such that the estimate (\ref{bounded from above and below}) holds.
\end{corollary}
\begin{proof}
The estimate (\ref{bounded from above and below}) implies that $\|\log w_{\Lambda}^{(y)}\|_{\infty}^{0}\leq \frac{1}{2}\log{(M/m)}$ and so
\[
[w_{\Lambda}^{(y)}]^{-2}_{A_2(HS)}\geq \frac{m}{M}.
\]
Putting this estimate in (\ref{hs rb bound}) the result follows.
\end{proof}

We give an example for an application of the last corollary.
\begin{example}\label{examp for vander}
Fix $K\in \N$ and let $\mathcal{X}:=\{x_1,x_2,...,x_K\}\subseteq [0,K]$. Denote by
\[
\Lambda:= \mathcal{X}+K\Z.
\]
the $K$-periodic set generated by $\mathcal{X}$. Then $E(\Lambda)$ is a Riesz basis in $L^2[0,1]$ with lower bound satisfying
\[
A(\L)\geq \frac{1}{7\delta}18^{-\frac{2K}{\delta}},
\]
where $\delta$ is the minimal distance between elements in $\mathcal{X}$ modulo $K$. The exponential rate of decay in $K$ is necessary.
\end{example}
\begin{proof}

To see this first note that the function $F_{\L}$, defined by the product (\ref{eq;product}), is equal up to a multiplicative constant to
\[
\prod_{k=1}^K  e^{\frac{\pi iz}{K}}\sin{\frac{\pi}{K}(z-x_k)}.
\]
Next, observe that
\[
2|e^{\frac{\pi}{K}iz}\sin \frac{\pi}{K}(z-x_k)|=|e^{\frac{2\pi}{K}i(z-x_k)}-1|,
\]
 and that
\[
1-e^{-\frac{2\pi}{K}y}\leq \big|e^{\frac{2\pi}{K}i(x+iy-x_k)}-1\big| \leq 1+e^{-\frac{2\pi}{K}y}.
\]
We obtain
\[
\frac{\inf_x|F_{\L}(x+iy)|^2}{\sup_x|F_{\L}(x+iy)|^2}\geq \Big(\frac{1-e^{-\frac{2\pi}{K}y}}{1+e^{-\frac{2\pi}{K}y}}\Big)^{2k}.
\]
For $y=(K\log3)/2\pi$ the last expression is equal $1/2^{2k}$. Inserting this to (\ref{sine type estimate}) and noting that $\delta\leq 1$ the estimate follows.

To see that an exponential rate of decay in $K$ is necessary, fix an integer $L\geq 2$ and let $K=4L$. Denote by $\L_L$ the $K$-periodic set which is generated by $\{L, L+1/2, L+1, L+3/2, ...3L-1/2\}$. Consider the function
\[
\phi_L(x):=\Big(\frac{L\sin \frac{\pi}{L}x}{\pi x}\Big)^L,
\]
and let $\Phi_L\in L^2[-1/2,1/2]$ be the inverse Fourier transform of $\phi_L$. Note that
\[
1=\phi_L(0)=\int_{-1/2}^{1/2}\Phi_L\leq \|\Phi_L\|_1\leq \|\Phi_L\|_2.
\]
Standard considerations give
\[
\begin{aligned}
A(\Lambda_L)&\leq \sum_{\l\in\L_L}|\langle \Phi_L, e^{2\pi i \l t}\rangle|^2= \sum_{\l\in\L_L}|\phi_L(\l)|^2\\
&\leq 2\Big(\frac{L}{\pi}\Big)^L\sum_{k=0}^{\infty}\frac{1}{(L+k/2)^L}\leq \Big(\frac{L}{\pi}\Big)^L\frac{4}{(L-1)L^{L-1}}\leq \frac{8}{\pi^{L}}.
\end{aligned}
\]
\end{proof}

\begin{remark}\label{not optimal I}
Example \ref{examp for vander} is presented here to demonstrate a technique and the estimate obtained in it is in general not optimal. See Remark \ref{not optimal II} for an improved estimate obtained by different means.
\end{remark}

In \cite{Avdonin} Avdonin proves that his perturbation result holds also for zero sets of sine type functions. This can be seen, for example, by noting that if $\L$ is such a zero set, then the phase function $\alpha_{\Lambda+iy}$ is - up to an additive constant - the Hilbert transform of a bounded function (this follows, for example, from  Remark \ref{relation between things}). For a similar reason, corresponding versions of Corollary \ref{cor; general stability} and Theorem \ref{thm:avdonin} hold in this case as well. We formulate the former but omit it's proof, as it is similar to the proof presented in the previous section.

\begin{corollary}\label{cor; general stability sine type}
Let $\Lambda\subseteq \R$ be uniformly discrete and assume that $\L$ is a zero set of a sine type function of full diagram for which the estimate (\ref{bounded from above and below}) holds with $y,m,M>0$. Let $\Gamma=\{\gamma_{\l}\}_{\l\in\L}$ be a uniformly discrete sequence for which $\mu_{\l}=\gamma_{\l}-\l$ satisfy
$
\sup_{\l\in\L}{|\mu_{\l}|}<\infty.
$
For $y>0$ denote
\[
\tau_y:=\sup_{x\in{\R}}\Big|{\sum_{\l\in\L}\int_{x}^{x+\mu_{\l}}\frac{y}{(\l -t)^2+y^2}\: dt}\big|.
\]
If there exists $y_0\geq y$ so that $\tau_{y_0}<\pi/4$ then $E(\Gamma)$ is a Riesz basis in $L^2[0,1]$ and its lower Riesz basis bound satisfies
%\[
%A(\Lambda)\geq \frac{3}{\pi}\frac{e^{-8\pi (\frac{y}{\delta})}}{1- e^{-y}}\cos^2{2\pi \tau_y}\quad \forall y\geq y_0.
%\]
%In particular this implies that
\[
A(\Lambda)\geq \frac{1}{7\delta(\Gamma)}\frac{m}{M}e^{-8\pi \frac{y_0}{\delta(\Gamma)}}\cos^2{2\tau_{y_0}}.
\]
\end{corollary}

In a similar way one may follow the outline of Theorem \ref{thm:avdonin} and its proof to obtain an estimate for the lower Riesz basis bound in Avdonin's theorem for zero sets of sine type functions.

\begin{remark}
It is well known that a version of Kadetz' theorem holds also for general exponential Riesz bases in $L^2[0,1]$: If $E(\Lambda)$ is such a basis then there exists $\mu>0$ so that if $\Gamma=\{\gamma_{\l}\}_{\l\in \L}\subseteq\R$ satisfies $|\gamma_{\l}-\lambda|<\mu$, for every $\l\in \L$, then $E(\Gamma)$ is a Riesz basis as well.

A similar extension holds also for Avdonin's theorem. Indeed, if $E(\Lambda)$ is a Riesz basis then by Theorem \ref{thm; exp rb and phase} there exist $y>0$, two bounded functions $v_y,u_y\in L^{\infty}(\R)$, with $\|v_y\|_{\infty}<1/4$, and a constant $c\in\mathbb{\R}$, so that
$
\alpha_{\L+iy}=v_y+\widetilde{u}_y+c.
$
Picking any $0<\mu<1/4-\|v_y\|_{\infty}$ will provide the result. The proof is similar to the proof of Avdonin's theorem and we omit the details.

In the above extension for Kadetz' theorem, an estimate for $A(\Gamma)$ in terms of $A(\Lambda)$ and $\mu$ can be given: If $\mu<\sqrt{A(\L)}\delta(\L)/8\pi$ then
\begin{equation}
A(\Gamma)\geq \Big(\sqrt{A(\L)}-\frac{8\pi\mu}{\delta(\L)}\Big)^2.
\end{equation}
(this can be deduced, for example, from the proof in \cite{KozNit}, page 4). We do not know to provide a similar estimate for the extension of Avdonin's theorem.
\end{remark}

\subsection{Finite dimensions}

In this section we show that a version of the above results holds in the finite dimensional setting. We do so by showing that results in the finite dimensional setting can be  directly obtained from their counterpart on the real line. A similar approach is taken in \cite{ACL} to study versions of Kadetz' theorem in this setting. See Remark \ref{acknoledgments} below for more details.

Let $d\in \N$ and consider the space
\[
\ell^2_d:=\{(a_0,a_1,..,a_{d-1}): a_j\in\mathbb{C}\: j=0,...,d-1\},
\]
equipped with the standard norm and inner product. For $\theta \in [0,1]$ denote by
\[
e_d(\theta)=\frac{1}{\sqrt{d}}(1, e^{2\pi i \theta}, e^{2\pi i \theta}, e^{2\pi i 2\theta},..., e^{2\pi i (d-1)\theta})\in \ell^2_d,
\]
the normalized discrete exponential function. Given a finite set $\Theta\subseteq [0,1]$, we are interested in cases where
$
\mathcal{E}(\Theta):=\{e_d(\theta)\}_{\theta\in\Theta}
$
is a Riesz basis in $\ell^2_d$. It is strait forward to verify that this happens if and only if $|\Theta|=d$. In this case the Riesz basis bounds satisfy
\[
A(\Theta)=\frac{1}{d}\|(V(\Theta))^{-1}\|^{-2}\quad\textrm{and}\quad B(\Theta) = \frac{1}{d}\|(V(\Theta))\|^{2},
\]
where $V(\Theta)$ is the $d\times d$ Vandermonde matrix generated by $\{e^{2\pi i \theta}\}_{\theta\in \Theta}$, that is,
\[
V(\Theta)(j,k)=e^{2\pi i k\theta_j}\qquad k,j=1,...,d.
\]
We may rewrite this last fact as follows: For every $a\in \ell^2_d$ we have
\begin{equation}\label{RB via VdM}
A(\Theta)\|a\|^2\leq \frac{1}{d}\|V(\Theta)a\|^2\leq B(\Theta)\|a\|^2,
\end{equation}
and these bounds are sharp.

It is well known that certain problems regarding exponential systems over subsets of the real line can be reformulated in terms of the finite system $\mathcal{E}(\Theta)$, or equivalently, in terms of the Vandermonde matrix $V(\Theta)$. As far as we know, such a relation was first introduced by Landau to study complete systems of exponentials  \cite{Landau}, and by Bezuglaya and Katsnel'son to study Riesz bases of exponentials \cite{Katznelson}. Several later works apply this approach, for example, it was used in \cite{OlUluniversal, NOlUl} to study Riesz sequences and frames of exponentials. 

In the context of our discussion, the following particular case is of interest. The proof follows along the lines of the proofs in \cite{Katznelson, NOlUl} and we omit it.
\begin{theorem}\label{line to finite}
Fix $d\in\N$ and let $\theta =\{\theta_j\}_{j=1}^{d}\in [0,1]$. Denote
\[
\Lambda=\bigcup_{j=1}^d(d\theta_j+dZ).
\]
Then $\mathcal{E}(\Theta)$ has the same Riesz basis bounds in $\ell^2_d$ as $E(\L)$ has in $L^2[0,1]$, that is
\[
A(\Lambda)=A(\Theta)=\frac{1}{d}\|(V(\Theta))^{-1}\|^{-2},\quad\textrm{and}\quad B(\Lambda)=B(\Theta) = \frac{1}{d}\|V(\Theta)\|^{2}.
\]
\end{theorem}

\begin{remark}\label{works also for frames}
The notion of \textit{Riesz basis} in Theorem \ref{line to finite} may be replaced by either one of the notions \textit{Bessel system, frame}, or \textit{Riesz sequence}. The corresponding Vandermonde matrix may be rectangular (See also \cite{NOlUl} in this regard).
\end{remark}

In the manuscripts mentioned above, relations of a similar form were used to obtain results for exponential systems over the line. Here we wish to work in the opposite direction, and use Theorem \ref{line to finite} to translate results obtained in previous sections into results in the finite dimensional setting. We start with a simple example.

\begin{example}\label{van der estimate}
Fix $d\in\N$ and let $\theta =\{\theta_j\}_{j=1}^{d}\in [0,1]$. Denote $\delta=\delta(\Theta)=\min_{j\neq k}|\theta_j-\theta_k|$ (mod 1). Then
\[
\|(V(\Theta))^{-1}\|\leq \frac{1}{7\sqrt{d}\delta}18^{\frac{2 }{\delta}}.
\]
\end{example}

\begin{proof}
This follows directly from combining Example \ref{examp for vander} and  Theorem \ref{line to finite}.
\end{proof}

\begin{remark}\label{not optimal II}
The result in Example \ref{van der estimate} can be obtained by different means, and in general is not optimal. Indeed, in \cite{BDGY} the authors apply a result of Gautschi \cite{Gau} to obtain the estimate:
\[
\|(V(\Theta))^{-1}\|\leq \sqrt{d}\pi^{d-1}\max_k\Pi_{j\neq k}\frac{1}{\delta_{jk}}\leq \frac{\sqrt{d}\pi^{d}}{d!\delta^d},
\]
where $\delta_{jk}=|\theta_j-\theta_k|$ (mod 1). Since $\delta\leq 1/d$ this estimate is in general better then the one given by  Example \ref{van der estimate} (as can be verified, say, via Stirling's approximation). Of course, one can apply Theorem \ref{line to finite} to this estimate and obtain a similar improvement of the estimate in  Example \ref{examp for vander} (see remark \ref{not optimal I}).
\end{remark}

For $a=(a_j)_{j=0}^{d-1}\in \ell^2_d$ let $p_a(z)=a_0+a_1z+...+a_{d-1}z^{d-1}$, and consider $p_a$ as a polynomial restricted to the torus $\mathrm{T}:=\R/\Z$ via the relation $z=e^{2\pi it}$. Note that $\|p_a\|_{L^2(\mathrm{T})}=\|a\|_{\ell^2_d}$. With this relation we identify $\ell^2_d$ with the space of polynomials
\[
\mathcal{P}_{d-1}:=\{p_a(z):\: a\in \ell^2_d\}.
\]
Let $\Theta\subset [0,1]$. The relation (\ref{RB via VdM}) implies that
$
\mathcal{E}(\Theta)
$
is a Riesz basis in $\ell^2_d$ if and only if $\Theta$ is a \textit{minimal sampling sequence} for $\mathcal{P}_{d-1}$ with the same bounds, that is, if and only if  $|\Theta|=d$ and
\begin{equation}\label{sampling}
A(\Theta)\|p\|_{L^2(\T)}^2\leq \frac{1}{d}\sum_{\theta\in\Theta}|p(e^{2\pi i \theta})|^2\leq B(\Theta)\|p\|_{L^2(\T)}^2\quad \forall p\in \mathcal{P}_{d-1}.
\end{equation}
We are interested in families of such minimal sampling sequences, for which the sampling bounds $A(\Theta)$ and $B(\Theta)$ are independent of the dimension $d$. 

We say that a family  $\Upsilon=\{\Theta_d\}_{d\in\N}$ is a \textit{triangular family} if for every $d\in\N$ we have $\Theta_d\subset [0,1]$, if in addition $|\Theta_d|=d$ we say that the triangular family is \textit{minimal}. We use the following convention.
\begin{definition}
Let $\Upsilon=\{\Theta_d\}_{d\in\N}$ be a minimal triangular family. We say that $\Upsilon$ is a \textit{minimal Marcinkiewicz-Zygmund} (MZ) family if there exist $A,B>0$ so that for every $d\in\Z$ the sequence $\Theta_d$ satisfies (\ref{sampling}) with $A(\Theta_d)\geq A$ and $B(\Theta_d)\leq B$. The largest $A$ and smallest $B$ for which these inequality hold are called the MZ lower and upper bounds. We denote them by $A(\Upsilon)$ and $B(\Upsilon)$.
\end{definition}

We will need the following definitions.

\begin{definition}
Let $\Upsilon=\{\Theta_d\}_{d\in\N}$ be a triangular family.

\vspace{5pt}

\begin{itemize}
\item[i.] We say that $\Upsilon$ is uniformly discrete if there exist $\delta>0$ so that for every $d\in\N$ and every $\theta_1,\theta_2\in\Theta_d$ we have
\[
|\theta_1-\theta_2|\geq \frac{\delta}{d} \textrm{ (mod 1)}.
\]
The largest $\delta$ for which this inequality holds is called the separation constant of $\Upsilon$. We denoted it by $\delta(\Upsilon)$.

\vspace{5pt}

\item[ii.] Let $\Sigma=\{\Delta_d\}_{d\in\N}$ be an additional triangular family, and let $\epsilon>0$. We say that $\Sigma$ is an $\epsilon$-perturbation of $\Upsilon$ if for every $d\in\N$ there exists a $1-1$ correspondence $\phi_d:\Theta_d\rightarrow\Delta_d$ so that
    \[
    |\theta-\phi_d(\theta)|\leq \frac{\epsilon}{d} \textrm{ (mod 1)}, \qquad \forall\theta\in\Theta_d.
    \]
\end{itemize}
\end{definition}

We first note the following (see also \cite{OrtegaCerda}).
\begin{proposition}\label{MZ bessel}
Let $\Upsilon=\{\Theta_d\}_{d\in\N}$ be a uniformly discrete minimal triangular family with separation constant $\delta(\Upsilon)$. Then for every $d\in\N$ the right hand side of  (\ref{sampling}) holds with
\[
 B(\Theta_d)\leq \frac{8\pi}{\min\{\delta(\Upsilon), 1\}}.
 \]
\end{proposition}
\begin{proof}
For $d\in \N$ put $\Theta_d=\{\theta_j(d)\}_{j=1}^d$ and let
\[
\Lambda_d=\bigcup_{j=1}^d(d\theta_j(d)+dZ).
\]
Then $\Lambda_d$ is uniformly discrete with separation constant $\delta(\Lambda_d)\geq \delta(\Upsilon)$ and so (\ref{exp bessel}) implies that
\[
B(\Lambda_d)\leq \frac{8\pi}{\min\{\delta(\Upsilon), 1\}}.
\]
By Theorem \ref{line to finite} we have $B(\Lambda_d)=B(\Theta_d)$. The result follows.
\end{proof}

The canonical Example of a minimal MZ family is given by
\begin{equation}\label{Omega}
\Omega=\{\Phi_d\},\qquad \Phi_d=\big\{\frac{j}{d}\big\}_{j=0}^{d-1}.
\end{equation}
 In this case the Vandermonde matrix $V(\Phi_d)$ is the classical $d\times d$ Fourier matrix, and the sampling points in (\ref{sampling}) are the $d$ roots of unity. It is straightforward to verify that for every $d$ we have $A(\Phi_d)=B(\Phi_d)=1$ and so $A(\Omega)=B(\Omega)=1$. Keeping   Theorem \ref{line to finite} in mind, we see that $\Omega$ plays the same role in the finite dimensional setting as $E(\Z)$ plays in the setting of $L^2[0,1]$.

In particular, it seems reasonable to expect that stability results which hold for $E(\Z)$ in $L^2[0,1]$ will have a version which holds for $\Omega$ in the finite dimensional setting. Indeed, in \cite{MarzoSeip} Seip and Marzo obtain a version of Kadetz' $1/4$ theorem for MZ families, and in \cite{YT} Yu and Townsend estimate the corresponding lower MZ bound. In \cite{Marzo} Marzo obtains a version of Avdonin's theorem for this setting. The proofs in all cases were rather involved.  Here we present a simpler approach, which nevertheless allows to estimate the lower MZ bound in both theorems. We start with a version of Kadetz' theorem.

\begin{theorem}
Let $\Omega=\{\Phi_d\}_{d\in\N}$ be as in (\ref{Omega}) and let $\mu<1/4$. If a triangular family $\Upsilon=\{\Theta_d\}_{d\in\N}$ is a $\mu$-perturbation of $\Omega$ then it is a minimal MZ family and its lower and upper MZ bounds satisfy
\begin{equation}\label{fin-dim-kadetz}
A(\Upsilon)\geq 2\sin ^2{\frac{\pi}{4}(1-4 \mu)};\qquad B(\Upsilon)\leq 8.
\end{equation}
The $1/4$ restriction in this result is sharp.
\end{theorem}
\begin{proof}
Fix $d\in \N$ and put $\Theta_d=\{\theta_j(d)\}_{j=1}^d$. Let
\[
\Lambda_d=\bigcup_{j=1}^d(d\theta_j(d)+dZ).
\]
Since  $\Upsilon$ is a $\mu$-perturbation of $\Omega$, there exists a numeration $\Lambda_d=\{\lambda_n\}_{n=-\infty}^{\infty}$ so that
$
|\lambda_n-n|<\mu$ for all $n\in\Z$.
The estimate (\ref{kadec bound}) now implies that
 \[
A(\L_d)\geq 2\sin ^2{\frac{\pi}{4}(1-4 \mu)}.
\]
By Theorem \ref{line to finite} we have $A(\Lambda_d)=A(\Theta_d)$. Since this estimate holds for every $d\in\N$ a similar estimate holds also for $A(\Upsilon)$. The estimate for $B(\Upsilon)$ now follows from Claim \ref{MZ bessel}.

The sharpness of the $1/4$ restriction was proved in \cite{MarzoSeip}, we present here a different proof. Consider the triangular family $\Delta=\{\Sigma_d\}$ defined by
\[
\Sigma_{2k+1}=\{\pm\frac{j-1/4}{d}\}_{j=1}^{k}\cup\{0\}\qquad \Sigma_{2k}=\{\pm\frac{j-1/4}{d}\}_{j=1}^{k-1}\cup\{0,\frac{k-1/4}{d}\},
\]
where the values are understood to be (mod1). Assume for a contradiction that $\Delta$ is an MZ family with lower bound $A(\Delta)>0$. For $d\in \N$ put $\Sigma_d=\{\sigma_j(d)\}_{j=1}^d$ and let
\[
\Lambda_d=\bigcup_{j=1}^d(d\sigma_j(d)+dZ).
\]
By Theorem \ref{line to finite}, for every $d\in\N$ the system $E(\Lambda_d)$ is a Riesz basis in $L^2[0,1]$ with a lower Riesz basis bound satisfying $A(\Lambda_d)\geq A(\Delta)$. It follows that for $\Gamma=\{\pm(n-1/4)\}_{n-1}^{\infty}\cup\{0\}$ the system $E(\Gamma)$ is a Riesz basis in $L^2[0,1]$ as well. Indeed, this can be shown by noting that $\Gamma$ is a weak limit of $\Lambda_d$ and applying Lemma's 3.13 and 3.14 in \cite{Seip}. However, it is well known that  $E(\Gamma)$ is a not a Riesz basis in $L^2[0,1]$ (see e.g, \cite{OlUlbook} Exersize 1.24 and the hint thereafter). We arrive at a contradiction.
\end{proof}

Similar to the case of exponential systems, a Kadetz-type result holds also for general minimal MZ families. The precise formulation is given below.

\begin{theorem}\label{finite general kadetz}
Let  $\Upsilon=\{\Theta_d\}_{d\in\N}$ be a minimal MZ family and \[\mu<\sqrt{A(\Upsilon)}\delta(\Upsilon)/8\pi.\] If a triangular family $\Sigma$ is a $\mu$-perturbation of $\Upsilon$ then $\Sigma$ is also a minimal MZ family and its lower MZ bound satisfies
\begin{equation}\label{fin-gen-kad}
A(\Sigma)\geq \Big(\sqrt{A(\Upsilon)}-\frac{8\pi\mu}{\delta(\Upsilon)}\Big)^2.
\end{equation}
\end{theorem}

\begin{remark}\label{rectcase}
Claim \ref{MZ bessel} and Theorem \ref{finite general kadetz} hold also for more general triangular families which are not necessarilly minimal. See Remark \ref{works also for frames} in this regards.
\end{remark}
\begin{remark}\label{acknoledgments}
While working on this manuscript we became aware that O.~ Aspichuk, L.~De-Carli, and W.~Li independently obtained the estimates (\ref{fin-dim-kadetz}) and (\ref{fin-gen-kad}) using similar methods. See a corresponding acknowledgment at the end of the introduction in \cite{ACL}. A more comprehensive treatment of the rectangular case referred to in Remark \ref{rectcase}, as well as extensions of the estimates (\ref{fin-dim-kadetz}) and (\ref{fin-gen-kad}) to higher dimensions, can be found in \cite{ACL} as well.
\end{remark}

Finally, we formulate a version of Avdonin's theorem for MZ families with an estimate on the lower MZ bound. Here too we find the proof to be essentially the same as the proofs presented above and we omit it. We will need the following notations: Let  $\Upsilon=\{\Theta_d\}_{d\in\N}$  be a triangular family where $\Theta_d=\{\theta_j(d)\}_{j=1}^d$. For $d\in \N$ and $j=0,...,d-1$ we denote
\[
\mu^{(d)}_j=\theta_j(d)-\frac{j}{d},
\]
and for $N\in\N$ we let
\begin{equation}
\rho(N):=\sup\Big|\frac{d}{N}\sum_{j=mN}^{(m+1)N-1}\mu^{(d)}_j\Big|,
\end{equation}
where the supremum is taken over $d\in\N$ and $m=0,...,d-1$, and the variable $j$ in the sum is understood to be taken modulo $d$.
\begin{theorem}
Let  $\Upsilon=\{\Theta_d\}_{d\in\N}$  be a uniformly discrete triangular family with separation constant $\delta:=\delta(\Upsilon)$ and let $\mu^{(d)}_j$ and $\rho(N)$ be as above. Assume that there exists $L>0$ so that for every $d\in\N$ and $j=0,...,d-1$ we have $|\mu^{(d)}_j|<L/d$. If there exists $N\in\N$ so that
$\rho(N)<1/4$ then $\Upsilon$ is a minimal MZ family and its lower MZ bound satisfies
\[
A(\Upsilon)\geq \frac{1}{7\delta(\Upsilon)}e^{- \frac{960\pi L^2N}{\delta(\Upsilon)(1-4\rho)^2}}\sin^2{\frac{\pi}{4}(1-4\rho)},
\]
where $\rho:=\rho(N)$.
\end{theorem}

\begin{remark}
Some additional results regarding uniform sampling and interpolation in the polynomial spaces $\mathcal{P}_{d}$ may be obtained via a similar technique. These include, for example, density results in the spirit of Beurling and Landau's theorems and their extensions. (See  \cite{OrtegaCerda} for some results of this nature, obtained using a different machinery).
\end{remark}

\bibliographystyle{amsplain}
%\bibliography{../../../@bibliotek/bibliotek}

\def\cprime{$'$} \def\cprime{$'$} \def\cprime{$'$} \def\cprime{$'$}
\providecommand{\bysame}{\leavevmode\hbox to3em{\hrulefill}\thinspace}
\providecommand{\MR}{\relax\ifhmode\unskip\space\fi MR }
\providecommand{\MRhref}[2]{%
  \href{http://www.ams.org/mathscinet-getitem?mr=#1}{#2}
}
\providecommand{\href}[2]{#2}

\end{document}